\documentclass[11pt,reqno, english]{amsart}

\usepackage{amsmath,amsthm,amssymb,amsfonts}
\usepackage{pdfsync}
\usepackage{mathtools}
\usepackage[colorlinks=true, linktocpage=true]{hyperref}
\hypersetup{citecolor=blue}
\usepackage{booktabs, comment}
\usepackage{todonotes}
\usepackage{bm}
\usepackage{stmaryrd}

\usepackage[all]{xy}
\usepackage{pdfsync}
\usepackage{amssymb}
\usepackage{color}
\usepackage{longtable}
\usepackage{stackengine}
\def\subrangle#1{\stackengine{5pt}{}{$\!\scriptstyle #1$}{U}{l}{F}{F}{L}}

\usepackage[margin=1.5in]{geometry}

\newcommand{\norm}[1]{\left\lVert#1\right\rVert}

\newcommand{\R}{\mathbb{R}}
\newcommand{\T}{{\bf \Theta}}

\newcommand{\s}{\mathbb{S}}
\newcommand{\C}{\mathbb{C}}

\newcommand{\W}{\mathbb{W}}
\newcommand{\Mp}{Mp}
\newcommand{\Sp}{Sp}

\numberwithin{equation}{section}
\newtheorem{theorem}[subsubsection]{Theorem}  
\newtheorem{corollary}[subsubsection]{Corollary}
\newtheorem{lemma}[subsubsection]{Lemma}         
\newtheorem*{lemma*}{Lemma}         
\newtheorem{proposition}[subsubsection]{Proposition}
\newtheorem*{proposition*}{Proposition}

\theoremstyle{definition}
\newtheorem{remark}[subsubsection]{Remark}
\newtheorem{definition}[subsubsection]{Definition}              
\numberwithin{equation}{subsection}

\title[Local theta correspondence as a Morita equivalence]{Equal rank local theta correspondence as a strong Morita equivalence}

\author[Mesland]{Bram Mesland}
\address{\normalfont{Mathematisch Instituut, Universiteit Leiden \\
Postbus 9512, 2300 RA Leiden, Netherlands}}
\email{b.mesland@math.leidenuniv.nl}

\author[\c{S}eng\"un]{Mehmet Haluk \c{S}eng\"un}
\address{\normalfont{School of Mathematics and Statistics\\
University of Sheffield,
Hounsfield Road, Sheffield, S3 7RH, UK}}
\email{m.sengun@sheffield.ac.uk}

\begin{document}
 
 \begin{abstract} Let $(G,H)$ be one of the equal rank reductive dual pairs $\left (Mp_{2n},O_{2n+1} \right)$ or $\left (U_n,U_n \right )$ over a nonarchimedean local field of characteristic zero. It is well-known that the theta correspondence establishes a bijection between certain subsets, say $\widehat{G}_\theta$ and $\widehat{H}_\theta$, of the tempered duals of $G$ and $H$. We prove that this bijection arises from an equivalence between the categories of representations of two $C^*$-algebras whose spectra are $\widehat{G}_\theta$ and $\widehat{H}_\theta$. This equivalence is implemented by the induction functor associated to a Morita equivalence bimodule (in the sense of Rieffel) which we construct using the oscillator representation. As an immediate corollary, we deduce that the bijection is functorial and continuous with respect to weak inclusion. We derive further consequences regarding the transfer of characters and preservation of formal degrees. \end{abstract}
\maketitle
\tableofcontents 
\allowdisplaybreaks
\section{Introduction}   
In this paper, we cast a new light onto equal rank tempered local theta correspondence by approaching it via the framework of the representation theory of $C^*$-algebras. As a result, we discover some fundamental new features and obtain conceptual new proofs for several known facts.

Local theta correspondence, founded by Roger Howe in the mid 1970's, is a major theme in the theory of automorphic forms and representation theory. In a nutshell, local theta correspondence establishes a bijection between certain sets of smooth irreducible representations of reductive groups $G$ and $H$ which form a \emph{dual pair}, that is,  $G$ and $H$ sit inside a large enough symplectic group in such a way that they form each others' centralisers. Roughly speaking, this bijection is obtained by considering how the so-called oscillator representation of the ambient symplectic group decomposes as a $G\times H$-representation. When the two groups have ``the same size'', the local theta correspondence enjoys several attractive properties, in particular, it preserves temperedness. It is this tempered correspondence in the equal rank case that will be the first ingredient of our paper. 

The second ingredient of our paper is the notion of strong Morita equivalence for $C^*$-algebras introduced by Mark Rieffel, again, in the mid 1970's, as part of his $C^*$-algebraic generalisation of Mackey's theory of induced representations of locally compact groups. Given two $C^*$-algebras $A$ and $B$, roughly put, an \emph{equivalence $A$-$B$-bimodule} $X$ is an $A$-$B$-bimodule which is equipped with an $A$-valued inner product and a $B$-valued inner product such that these inner products satisfy certain compatibility and continuity conditions. If such a bimodule exists, then $A$ and $B$ are said to be {\em strongly Morita equivalent}. This is an equivalence relation. 

Given a representation $\pi$ of $B$ realized on a Hilbert space $V$ and an equivalence $A$-$B$-bimodule $X$, one can ``induce'' it to a representation $\textnormal{\small Ind}_A^B(X,\pi)$ of $A$ captured on the Hilbert space $X \otimes_B V$ obtained by \emph{interior tensor product} (this process is sometimes called {\it Rieffel induction}). This association is functorial and has an inverse implemented by the dual module of $X$, thus leading to an equivalence of categories of representations of $A$ and $B$. It identifies the lattices of two-sided closed ideals of $A$ and $B$ and furthermore, preserves weak containment and direct integrals.

\subsection{Description of the main result} We bring together the two themes above in the case where $(G,H)$ is an equal rank dual pair of the form $\left (Mp_{2n},O_{2n+1} \right)$ or $\left (U_n,U_n \right )$ over a nonarchimedean local field of characteristic zero. In this case, the theta correspondence $\pi \mapsto \theta(\pi)$ establishes a bijection between certain subsets of the tempered duals of $G$ and $H$. Let us name these subsets $\widehat{G}_\theta$ and $\widehat{H}_\theta$. 

We consider the reduced $C^*$-algebras associated to the groups $G$ and $H$. These are algebras of operators, going back to Irving Segal, which are obtained from the convolution action of the $L^1$-algebra of a locally compact Hausdorff group on its $L^2$-space. As such, these $C^*$-algebras are directly related to tempered representations. We exhibit ideals $C^{*}_{\theta}(G)$ and $C^{*}_{\theta}(H)$ of the reduced $C^*$-algebras of $G$ and $H$ whose spectra are homeomorphic to $\widehat{G}_\theta$ and $\widehat{H}_\theta$ respectively. 

We proceed to show that the (smooth) oscillator representation of $G{\times}H$ provides a natural bimodule for the reduced $C^*$-algebras of $G$ and $H$ and that this bimodule can be promoted to an equivalence $C^{*}_{\theta}(G)$-$C^{*}_{\theta}(H)$-bimodule in the sense above. Remarkably, the crucial compatibility property between the $C^{*}_{\theta}(G)$-valued and $C^{*}_{\theta}(H)$-valued inner products turns out to be precisely the so-called {\em local Rallis inner product formula} of Gan and Ichino \cite{Gan-Ichino-14}. 

We call the equivalence $C^{*}_{\theta}(G)$-$C^{*}_{\theta}(H)$-bimodule above the {\em oscillator bimodule} and denote it $\T$. The key point is that given an irreducible representation of $C^{*}_{\theta}(H)$, which is the same as an element of $\widehat{H}_\theta$, the induced representation $\textnormal{\small Ind}_{C^{*}_{\theta}(H)}^{C^{*}_{\theta}(G)}(\T,\pi)$ of $C^{*}_{\theta}(G)$ is the (integrated form) of the $G$-representation $\theta(\pi^*)$ where $\pi^*$ is the contragradient of $\pi$. In fact, $\textnormal{\small Ind}_{C^{*}_{\theta}(H)}^{C^{*}_{\theta}(G)}(\T,\pi)$ is precisely the (integrated form) of the $G$-representation obtained from $\pi$ via the influential ``averaging of matrix coefficients'' construction of Jian-Shu Li introduced in \cite{Li-89} in the so-called {\it stable range} case (roughly speaking, when $G$ is at least twice the size of $H$). In the equal rank cases, it is known that Li's construction, hence ours, agrees with $\theta(\pi^*)$. 

\subsection{} At this point, we can point out some immediate implications on theta correspondence that seem to be previously unknown to the best of our knowledge. We see that the tempered theta correspondence, in the equal rank set-up, is simply the restriction of an equivalence of categories of representations of two $C^*$-algebras to the irreducible objects. As such it is functorial. Moreover, as it is implemented by an equivalence bimodule, it enjoys various properties, such as the preservation of weak containment. In particular, the tempered theta correspondence is a homeomorphism between $\widehat{G}_\theta$ and $\widehat{H}_\theta$. 
 
\subsection{} The oscillator bimodule interpolates the oscillator representation of one group with the regular representation of the other group. This immediately implies that $\widehat{G}_\theta$ lies in the support of the oscillator representation viewed as a $G$-representation. In fact, we show with an elementary analysis that the latter is precisely the closure of the former. This also follows from a Plancherel decomposition result of Sakellaridis \cite{Sakellaridis-17}.

\subsection{} As discussed in \cite{Rieffel-76}, strong Morita equivalence is intimately related to the so-called generalized commutation relations. Our oscillator bimodule construction immediately implies, via Theorem 1.9 of \cite{Rieffel-76}, that $G$ and $H$ generate each others' commutants (in the sense of von Neumann) inside the algebra of bounded linear operators of the Hilbert space carrying the oscillator representation. In \cite[Thm. 6.1]{Howe-89}, Howe proves this for general real dual pairs. 

\subsection{} The oscillator bimodule $\T$ can be viewed as generalization of the {\em Heisenberg module} of Rieffel which plays an important role in the theory of non-commutative tori and also features in Gabor analysis. Let $W$ be a symplectic vector space. Given a closed subgroup $\Gamma$ of $W$, Rieffel shows in \cite{Rieffel-88}  that the (twisted) $C^*$-algebras associated to 
$\Gamma$ and its dual/annihilator group $\Gamma^\perp$ are strongly Morita equivalent. The equivalence is implemented by the Heisenberg module based on the (projective) Heisenberg representation of $W$. The critical compatibility condition for the two operator valued inner products on the bimodule reduces to the Poisson transformation in this case. 

In the local theta correspondence set-up, we operate inside the symplectic group $Sp(W)$ with the roles of $\Gamma, \Gamma^\perp$ played by the equal rank dual pair $(G,H)$. Accordingly we consider not the Heisenberg representation but the more complicated oscillator representation. In this sense, the local Rallis inner product formula of Gan and Ichino that we used in our proof of the aforementioned compatibility property can be viewed as a non-commutative analogue of the Poisson transform. 

\subsection{Applications} After promoting the equal rank tempered theta correspondence to a categorical equivalence, we move on to illustrate the fact that an equivalence bimodule allows the transfer of information between the two sides.  We do this with two applications that are attractive in the simplicity of their statements and the elementary nature of their short proofs. 

\subsubsection{Explicit transfer of characters} \label{char} Let us continue with the set-up of the above section. If $\pi$ is a tempered irreducible representation of $H$, the character of $\pi$ is the tempered distribution on $H$, that is, the continuous linear functional 
$${\rm ch}(\pi) : \mathcal{S}(H) \to \C$$
on Harish-Chandra's Schwartz algebra $\mathcal{S}(H)$ of $H$ given by the trace 
$${\rm ch}(\pi)(\varphi):= {\rm tr} \ \pi(\varphi).$$

The oscillator bimodule forms a connection between parts of the Schwartz algebras of $G$ and $H$, and as such gives a meaningful way of expressing the character one representation in terms of that of its theta lift. An elementary half-a-page long argument based on a concrete representation of the oscillator bimodule $\T$ as a space of operators gives us the following.
\begin{corollary} Let $\pi$ be a tempered irreducible representation of $H$ that enters the theta correspondence. Let $\s$ denote the Fr\'echet space carrying the smooth oscillator representation of $G{\times}H$. Given $x,y \in \s$, let $\langle x,y \rangle\subrangle{H} \in \mathcal{S}(H)$ and $\prescript{}{G}{\langle x,y \rangle} \in \mathcal{S}(G)$ be the matrix coefficient functions defined below in (\ref{right-inner-product}) and (\ref{left-inner-product}). We have
 $${\rm ch}(\theta(\pi))(\prescript{}{G}{\langle x,y \rangle}) = {\rm ch}(\pi)(\langle y,x \rangle\subrangle{H}).$$
\end{corollary}

The inner products $\prescript{}{G}{\langle {\cdot},{\cdot} \rangle}$ and $\langle {\cdot},{\cdot} \rangle\subrangle{H}$ span ideals in the Schwartz algebras of $G$ and $H$ respectively and the above result explicitly relates the two characters when they are restricted to these ideals. To extend the above transfer formula beyond these ideals, various convergence issues in the theory of operators on Hilbert $C^*$-modules need to be addressed. We do not pursue this.\footnote{In the case of the ortho-symplectic pair, if one works with $SO(V)$ instead of $O(V)$, it is likely that our construction still goes through and in this case the ideal would be the whole Schwartz algebra of $SO(V)$.}.

Investigations on the question of how characters of representations relate, if at all, under the theta correspondence go back to the late 1980s. Notably, Przebinda studied the stable range case\footnote{He also proposed an explicit formula which is conjectured  to hold beyond the stable range case.} over the reals (e.g. \cite{Przebinda-91, Przebinda-00}). For a more recent result in this direction, see \cite{Merino-20} which treats real cases in which one group is compact and \cite{Loke-Przebinda-22} which adapts earlier works of Przebinda to the non-archimedean stable range setting. Our result in the non-archimedean equal rank case has been announced by Wee Teck Gan in a few talks in the recent years (see \cite{Gan-20, Gan-21}). Gan's proof seems to be different from ours, although both make essential use of matrix coefficients. Our proof is a simple and direct consequence of the Hilbert $C^*$-module point of view that we take in this paper. 

\subsubsection{Preservation of formal degrees} 
It is well-known (see \cite{Gan-Savin-12}) that the local theta correspondence takes discrete series representations to discrete series representations in the setting of equal rank pairs. Recall that the {\em formal degree}  of a discrete series representation $\pi$ of, say, $H$ is the positive real number ${\rm deg}(\pi)$ such that 
$$\int_H \langle v, \pi(h)(v') \rangle \overline{\langle w, \pi(h)(w') \rangle}\mathrm{d}s= \dfrac{1}{{\rm deg}(\pi)} \langle v, w \rangle \langle v', w' \rangle$$
for all $v,v',w,w' \in V_{\pi}$. It depends on the chosen Haar measure on $H$.

Using a cohomological argument mixed with some known facts regarding the transfer of trace maps under equivalence bimodules, we obtain the following.
\begin{corollary} Let $\pi$ be a discrete series representation of $H$ which enters the theta correspondence. Then 
$${\rm deg}(\pi) = {\rm deg}(\theta(\pi)).$$
\end{corollary}
The Haar measures used in the above result are the ones that we use for the proof the compatibility property (Prop. \ref{poisson-proof}) of the oscillator bimodule. The key point of the proof is that discrete series give generators of $K_0$ of the reduced group $C^*$-algebra and one can access their formal degrees using the canonical trace. The canonical trace is given by the orbital integral associated to the trivial conjugacy class. We did not pursue this as we did not need it, but one could also explicitly transfer traces  arising for orbital integrals of other conjugacy classes.

Although the above result on preservation of formal degrees result is not new  (Gan and Ichino proved it in \cite{Gan-Ichino-14}), we think that our proof is of interest as a simple application of $K$-theory, which is now readily available. This demonstrates the usefulness of the oscillator bimodule approach for the study of theta correspondence.

\subsection{Remarks} 
\begin{enumerate} 
\item Our approach also applies to the stable range case which was mentioned earlier. This case will be treated in a future paper. It should be possible to treat the almost equal rank cases as well.
\item We should point out that an equivalence of categories in the naive sense does not hold at the level of full smooth (as opposed to tempered) equal rank theta correspondence; this was explained to us by Dipendra Prasad who recently has been pursuing the idea of interpreting the full smooth equal rank theta correspondence as a `derived equivalence', see \cite{Prasad-22}. 
\item As pointed out to us by Wee Teck Gan, our $C^*$-algebraic approach to theta correspondence could be interpreted to lie within the general $C^*$-algebraic framework for symplectic quantization theory \cite{Landsman-94} developed by Klaas Landsman in the 1990's, see \cite{Gan-22}. 
\end{enumerate}
\subsection{Acknowledgments} We thankfully acknowledge helpful correspondences with Alexandre Afgoustidis, Pierre Clare, Nigel Higson, Allan Merino, Roger Plymen, Dipendra Prasad, Maarten Solleveld and Hang Wang. Special thanks go to Roger Howe for his encouragement and insightful comments in the early stages of our project. We are grateful to Wee Teck Gan for several illuminating and encouraging correspondences and conversations. In particular, it was Gan who brought the local Rallis inner product formula to our attention as a potential tool to prove the compatibility of the two inner products on the oscillator bimodule. 

We thank the Erwin Schr\"odinger Institute for its hospitality during the event ``Minimal Representations and Theta Correspondence'' in April 2022, and the Institute for Mathematical Sciences of the University of  Bath for its hospitality during the LMS-Bath Symposium ``$K$-theory and Representation Theory'' in July 2022. Part of the research in this paper was carried out within the online research community ``Representation Theory and Noncommutative Geometry'' sponsored by the American Institute of Mathematics. The second author gratefully acknowledges the invaluable support of the EPSRC New Horizons grant EP/V049119/1 which provided the much needed research time to develop this project. 

Finally, we are grateful to the referee for an exceptionally careful reading of the manuscript which led to this much improved version.


\section{The local theta correspondence} \label{LTC} Let $F$ be a non-archimedean local field of characteristic $0$. Let $E$ be $F$ or a quadratic extension of $F$. Put $\varepsilon=\pm 1$ and set
$$\varepsilon_0 =\begin{cases} \varepsilon, \quad \textrm{if} \ E=F, \\ 0,  \quad \textrm{if} \ E\not= F. \end{cases}$$
Following the conventions of \cite{Gan-Takeda-16}, we set 
$$W=W_n= \textrm{a $-\varepsilon$-Hermitian space over $E$ of dimension $n$},$$
$$V=V_m= \textrm{a $\varepsilon$-Hermitian space over $E$ of dimension $m$}.$$
We define the associated groups as follows:
$$G=G(W)=\begin{cases}\textrm{$Mp(W)$, \  if $W$ is symplectic and ${\rm dim}(V)$ is odd}, \\ 
\textrm{the isometry group of $W$, otherwise.}
\end{cases}
$$
Here $Mp(W)$ is {\em metaplectic group}: the unique nonsplit double cover of $Sp(W)$. We define $H=H(V)$ similarly by switching the roles of $W$ and $V$. If $E=F$ and $\varepsilon=1$, then $G=Sp(W)$ or $Mp(W)$ depending on the parity of the dimension of $V$. If $E=F$ and $\varepsilon=-1$, then $G=O(W)$. If $E \not= F$, then $G=U(W)$.

\subsection{The Heisenberg representation} Let ${\bf W}$ denote the space $W \otimes V$ equipped with the symplectic form 
$${\rm tr}_{E/F}\left ( ({\cdot},{\cdot})_W ({\cdot},{\cdot})_V \right ).$$
The {\em Heisenberg group} $H({\bf W})$ is defined as ${\bf W}\oplus F$ with the multiplication rule
$$(w,t){\cdot}(w',t') := (w+w', t+t'+\tfrac{1}{2}\langle w,w'\rangle).$$

We fix a non-trivial unitary character $\chi : F \to \C^1$. By the Stone-von Neumann Theorem, there exists, up to unitary equivalence, a unique irreducible unitary representation of $H({\bf W})$ with central character $\chi$. We denote this representation by $\rho_\chi$. 

\subsubsection{The oscillator representation} The group $Sp({\bf W})$ of isometries of the symplectic space ${\bf W}$ acts on the Heisenberg group $H({\bf W})$ as automorphisms via the rule $g{\cdot}(w,t):=(gw,t)$. Let $\mathcal{M}p({\bf W})$ denote the group satisfying the exact sequence
$$1 \to \C^1 \to \mathcal{M}p({\bf W}) \to Sp({\bf W}) \to 1.$$

There is a unique, up to equivalence, unitary representation $\omega_\chi$ of $\mathcal{M}p({\bf W})$ on the Hilbert space of $\rho_\chi$ satisfying the covariance property
\begin{equation} \label{covariance} \omega_\chi(\bar{g})\rho_\chi(h)\omega_\chi(\bar{g}^{-1})= \rho_\chi(g{\cdot}h)
\end{equation}
for all $h \in H({\bf W})$ and $g \in Sp({\bf W})$ with lift $\bar{g} \in \mathcal{M}p({\bf W})$. This representation is called the {\em oscillator representation}. 
 
\subsection{Splitting} Consider the natural map $G\times H \to Sp({\bf W})$. With the aid of a pair of auxillary characters $\chi_V, \chi_W$ of $E^{\times}$ (see \cite[3.3]{Gan-Ichino-14}), we can construct a splitting $\iota$
$$\xymatrix{ & \mathcal{M}p({\bf W}) \ar[d] \\ 
G {\times} H \ar@{.>}[ur]^{\iota} \ar[r] & Sp({\bf W})}
$$
We pull-back the local oscillator representation $\omega_{\chi}$ of $\mathcal{M}p({\bf W})$ to $G{\times}H$ via this splitting. We will mainly consider the underlying smooth representation and denote it simply by $\omega$, suppressing its dependency on $\chi, \chi_V, \chi_W$ for convenience.

\subsection{The Theta lift} Given a smooth representation $(\pi,V_\pi)$ of $H$ (always assumed to be of finite length), the maximal $\pi$-isotypic quotient of $(\omega,V_\omega)$ has the form
$$\pi \otimes \Theta(\pi)$$
for some smooth representation $\Theta(\pi)=\Theta_{W,V}(\pi)$ of $G$, known as the {\bf big theta lift} of $\pi$. Alternatively, we can describe $\Theta(\pi)$ as the representation $\omega \otimes {\bf 1}$ of $G$ on the space of $H$-coinvariants
\begin{equation} \label{eq: big-theta} \left ( V_\omega \otimes V_{\pi^*} \right )_{H}, 
\end{equation} 
where $(\pi^*,V_{\pi^*})$ is the contragradient  of $\pi$. The maximal semisimple quotient of $\Theta(\pi)$ is denoted by $\theta(\pi)=\theta_{W,V}(\pi)$ and is called the {\bf small theta lift} of $\pi$. 

In the case $H=Mp(W)$, we call $\pi$ genuine if it does not factor through $Sp(W)$. If $\pi$ is not genuine, it is easy to see that its big theta lift is zero. We now state the fundamental result for local theta correspondence theory.
\begin{theorem} (Howe Duality) 
If $\Theta(\pi)$ is non-zero then it has a unique irreducible quotient, so that $\theta(\pi)$ is irreducible. 
Moreover, if $\theta(\pi)\simeq \theta(\pi')$ is non-zero then $\pi \simeq \pi'$. 
\end{theorem}
This was originally conjectured by Howe \cite{Howe-79} and proven by him \cite{Howe-89} in the archimedean setting. In our nonarchimedean set-up, it was proven by Waldspurger \cite{Waldspurger-90} when the residue characteristic $p$ was not equal to 2. Much later, Gan and Takeda \cite{Gan-Takeda-16} proved this for all $p$. 

\subsection{Equal rank correspondence} We will now specialize the discussion to {\bf equal rank pairs} $(G,H)$, that is, pairs for which we have $m=n+\varepsilon_0$. Precisely, these are 
$$(Mp_{2k},O_{2k+1}), (O_{2k+1},Mp_{2k}), (U_k,U_k).$$

\begin{theorem} \label{thm: big-theta} Assume that $m=n+\varepsilon_0$. Let $\pi$ be a tempered irreducible representation of $H$. If $\Theta(\pi)$ is not zero, then it is irreducible (thus $\Theta(\pi)=\theta(\pi)$) and tempered. 
\end{theorem}
\begin{proof} This can be found in \cite{Gan-Takeda-16-2}: see their Thm 1.2 and Lemma 4.1. 
\end{proof}

In passing, we mention that the non-vanishing of the above theta lifts has an elegant characterisation in terms of the standard $\varepsilon$-factors (see \cite[Thm. 11.1]{Gan-Ichino-14}).

\subsection{Li's form} \label{Li-construction} Let $(\pi,V_\pi)$\footnote{When we do not need to be precise about the carrier space $V_\pi$ of a representation $(\pi,V_\pi)$, we will suppress it from the notation.} be a {\em tempered} irreducible representation of $H$. Following Li \cite{Li-89}, one introduces a sesquilinear form $({\cdot}, {\cdot})_\pi$ on $V_\omega \otimes V_\pi$ as follows
\begin{equation} \label{Li-form} ( \phi \otimes v, \phi' \otimes v')_\pi := \int_{H} \langle \phi, \omega(h)(\phi') \rangle \langle v, \pi(h)(v') \rangle \mathrm{d}h
\end{equation}
The defining integral is well-known to be absolutely convergent in our equal rank case; see \cite[Cor. 3.2]{Li-89}. One can also conclude its convergence from that of the doubling zeta integral at $s=0$, see \cite[Lemma 9.5 (ii)]{Gan-Ichino-14}.

Using the fact that $H$ is unimodular, it is easy to see that this form is Hermitian and $G$-invariant with respect to the natural action $\omega \otimes {\bf 1}$ of $G$. 
Let $N$ denote the radical of $({\cdot},{\cdot})_\pi$, namely 
$$N:=\left  \{ \Phi \in V_\omega \otimes V_\pi \mid (\Phi,\Psi)_\pi=0 \ \ \forall \Psi \in  V_\omega \otimes V_\pi \right \}.$$
Then $N$ is stabilised by $G$ and thus the quotient 
\begin{equation}\label{Li-rep} \left (  V_\omega \otimes V_\pi \right )/N
\end{equation}
carries a unitary $G$-representation that we will denote by $L(\pi)$. 

A straightforward calculation shows that the subspace spanned by elements of the form 
$$\Phi-(\omega \otimes \pi)(h)(\Phi)$$ 
with $\Phi \in  V_\omega \otimes V_\pi$ and $h \in H$ lies inside the radical $N$. Therefore we have an projection 
$$V_{\Theta(\pi^*)} \simeq (V_\omega \otimes V_\pi)_{H} \twoheadrightarrow V_{L(\pi)}.$$

\subsection{} The next result is well-known to specialists. In fact, most of it can be found in the literature, alas not completely and not in the way we want. So we give a quick proof. 
\begin{proposition} \label{form-positive} Assume that $m=n+\varepsilon_0$. Let $\pi$ be a tempered irreducible representation of $H$. We have 
$$L(\pi) \simeq \theta(\pi^*).$$
Moreover, if the form $({\cdot}, {\cdot})_\pi$ is non-zero, then 
$$({\cdot}, {\cdot})_\pi \geq 0.$$ 
\end{proposition}

\begin{proof} Recall from Theorem \ref{thm: big-theta} that $\Theta(\pi^*)=\theta(\pi^*)$. Assume that $({\cdot}, {\cdot})_\pi$, and hence $L(\pi)$, is non-zero. Since $\pi$, and hence $\pi^*$, is tempered, by Thm. \ref{thm: big-theta}, we know that $\Theta(\pi^*)$ is either zero or is irreducible. Therefore $L(\pi)$ being a nonzero quotient of $\Theta(\pi^*)$ immediately implies that $L(\pi) \simeq \Theta(\pi^*)$. 

Now assume that $({\cdot}, {\cdot})_\pi$ is zero. We will show that $\Theta(\pi^*)$ is zero. This is proven\footnote{They work with 
${\rm Hom}_{U(V)}(\omega\otimes \pi,{\bf 1})$. Note that ${\rm Hom}_{U(V)}(\omega\otimes \pi,{\bf 1})\simeq {\rm Hom}( \left ( \omega\otimes \pi \right )_{U(V)},{\bf 1})\simeq {\rm Hom}(\Theta(\pi^*),{\bf 1})$.} in \cite[Prop. B.4.1]{Harris-Li-Sun} for the unitary pairs. Their proof adapts easily to the metaplectic/orthogonal pairs as well; indeed, this is essentially done in \cite[Prop. 16.1.3 (iii)]{Gan-Ichino-14} which treats both cases simultaneously. In \cite[Prop. 16.1.3 (iii)]{Gan-Ichino-14}, the authors consider only discrete series $\pi$ but the proof still works if $\pi$ is tempered as we indicate now. The first thing to point out is that, in the notation of \cite{Gan-Ichino-14}, the submodule $R(V,\chi_W) \oplus R(V',\chi_W)$ equals all of the degenerate principal series $I_{{\bf P}}^{{\bf H}}(0,\chi_V)$ (see \cite[Prop. 7.2.(i)]{Gan-Ichino-14}, compare with \cite[Prop. B.3.2]{Harris-Li-Sun}).  As mentioned above, the doubling zeta integral $\mathcal{Z}$ is convergent at $s=0$. It is also well-known (see \cite[Thm. 9.1.(iii)]{Gan-Ichino-14}) that $\mathcal{Z}$ is non-trivial on $I_{{\bf P}}^{{\bf H}}(0,\chi_V)$. Our assumption that $({\cdot}, {\cdot})_\pi$ is zero on $V_{\omega_{V,W}} \otimes V_\pi$ implies that $\mathcal{Z}$ is zero on the submodule $R(V,\chi_W)$, therefore it is non-zero on the complement $R(V',\chi_W)$. This implies that $({\cdot},{\cdot})_\pi$ is non-zero on $V_{\omega_{V',W}} \otimes V_\pi$. From the previous paragraph, it follows then that $\Theta_{V',W}(\pi^*)$ is nonzero. Now the theta dichotomy principle (see \cite[Cor. 9.2]{Gan-Ichino-14}) tells us that $\Theta_{V,W}(\pi^*)$ is zero as claimed. 

For the second claim, assume again that $({\cdot}, {\cdot})_\pi$ is non-zero. Non-negativity of $({\cdot}, {\cdot})_\pi$ follows immediately from Thm A.5 of \cite{Harris-Li-Sun}. One sets the groups $G$ and $H$ in the statement of Thm. A.5 to be equal to our $G$. Similarly, one sets the representations $\pi_H$ and $\pi_G$ in the statement of Thm. A.5 to be equal to our $\omega$ and $\pi$ respectively. The hypothesis (i) of Thm. A.5 is automatically satisfied since in the non-archimedean set-up all smooth vectors are $K$-finite, and hypothesis (ii) is also satisfied thanks to the fast decay of the matrix coefficients of the oscillator representation that we alluded to above, see Lemma \ref{mat-coeffs} below. Now non-negativity of $({\cdot}, {\cdot})_\pi$ is precisely the conclusion of Thm. A.5. One could also directly refer to \cite[Prop. 3.3.1]{Sakellaridis-17} for non-negativity. 
\end{proof}

\section{Some $*$-algebras associated to groups} 
In this section we discuss various topological algebras associated to a locally compact group. One of those algebras will be the reduced $C^{*}$-algebra $C^{*}_{r}(G)$, whose spectrum coincides with the tempered dual of $G$. Another one will be the Schwartz algebra $\mathcal{S}(G)$ is a dense subalgebra  $C^{*}_{r}(G)$ consisting of functions on $G$ and it is more susceptible to the explicit constructions and calculations that we present in Section \ref{Morita}.

\subsection{$C^*$-algebras of groups} \label{C*-algebras}
Given a locally compact Hausdorff topological group $G$, we let $L^1(G)$ denote the Banach $*$-algebra of integrable functions. It is well-known that there is a bijection between unitary representations of $G$ and non-degenerate $*$-representations of $L^1(G)$: given a (strongly continuous) unitary representation $\pi:G\to U(V_\pi)$ of $G$ on a Hilbert space $V_{\pi}$, we obtain a $*$-representation of $L^1(G)$ (still denoted $\pi$) by integrating
\begin{equation}\label{eq: integrated-form} \pi(f) := \int_G f(s) \pi(s) \mathrm{d}s,\end{equation}
where $f \in L^1(G)$. 

Let $(\pi,V_{\pi})$ be a unitary representation of $G$. We denote the $C^*$-algebra generated\footnote{This is the closure of the image of $L^1(G)$ with respect to the operator norm.} by the image of $L^1(G)$ under the $*$-representation $\pi: L^1(G) \to \mathcal{L}(V_{\pi})$ by
$$C^*_\pi(G).$$ 
It is called the {\em $C^*$-algebra of $G$ associated to $\pi$}. One of the most important examples is when we take $\pi$ to be the regular representation of $G$ on $L^2(G)$; in that case $C^*_\pi(G)$ is the so-called {\bf reduced $C^*$-algebra of $G$} and it has the established notation 
$$C^*_r(G).$$ 

Recall that the spectrum of $L^1(G)$ is in bijection with the unitary dual $\widehat{G}$ of $G$. As is well-known, $\widehat{G}$ comes equipped with a topology that is typically described via uniform approximation of matrix coefficients on compacta. On the other hand, there is a natural topology on the spectrum of $C^*_\pi(G)$ (see Section \ref{sec: continuity} for a brief discussion). It can be shown that the above bijection gives a homeomorphism between the spectrum of $C^*_\pi(G)$ and the {\em support of $\pi$}, denoted ${\rm supp}(\pi)$, that is, the subset of $\widehat{G}$ whose elements are the representations which are weakly contained in $\pi$. A special case of this is the well-known fact that the spectrum of $C^*_r(G)$ is homeomorphic to the ``reduced dual'', that is, the support of the regular representation of $G$:
\begin{equation} \widehat{C^*_r(G)} \xleftrightarrow{\text{homeom.}} \widehat{G}_{\textnormal{red}}.
\end{equation}

\subsection{Schwartz algebra} \label{HCSA}
Given a connected reductive linear algebraic group ${\bf G}$ over $F$ (which, we recall, is non-archimedean), let us put $G={\bf G}(F)$. We fix a minimal parabolic subgroup $P$ and a 
``good''\footnote{Rougly put, it needs to be the stabilizer of a well-chosen vertex in the building associated to $G$. See \cite[p. 544]{Gan-Ichino-14}} maximal compact subgroup $K$ so that $G=PK$. Consider the smooth normalized induced representation $I_P^G({\bf 1})$ of the trivial representation  of $P$ to $G$. Let $e_K$ denote the unique vector in $I_P^G({\bf 1})$ such that $e_K(k)=1$ for all $k \in K$. We define {\em Harish-Chandra's function} as the diagonal matrix coefficient of $v_K$:
$$\Xi(g):=\langle I_P^G({\bf 1})(g)e_K, e_K \rangle, \quad g \in G.$$ 
It is well-known that $\Xi$ is a positive, $K$-biinvariant function that satisfies $\Xi(g)=\Xi(g^{-1})$ (see \cite[Section II.1]{Waldspurger-03}). 

We say that a continuous function $f: G \to \C$ is {\bf rapidly decreasing} if for all $n >0$ we have 
\begin{equation}\label{SGseminorms}
v_n(f):=\sup_{g \in G} \  |f(g)| \ \Xi(g)^{-1} (1+\log \norm{g})^n \ <\infty,
\end{equation}
where $\norm{{\cdot}}$ is the standard norm on $G$ arising from a good choice of embedding $\iota: G \hookrightarrow GL_m(F)$ where $m$ is the $F$-rank of $G$ (see, e.g. \cite[p. 544]{Gan-Ichino-14}). The space $\mathcal{S}(G)$ of all rapidly decreasing, uniformly locally constant\footnote{A function is {\em uniformly locally constant} if it is $B$ bi-invariant for a compact open subgroup $B$.} functions on $G$ is an algebra under convolution and it is called the (Harish-Chandra) {\bf Schwartz algebra} of $G$. 

Given a compact open subgroup $K$, let $\mathcal{S}(G{\sslash}K)$ denote the subspace of functions in $\mathcal{S}(G)$ which are constant on the double cosets of $K$. Then the space $\mathcal{S}(G{\sslash} K)$ is a nuclear, unital Fr\'echet $*$-algebra under convolution, with the topology given by the seminorms $v_n$ in (\ref{SGseminorms}). We have that 
$$\mathcal{S}(G) = \bigcup_{K} \mathcal{S}(G{\sslash} K)$$
where $K$ ranges over compact open subgroups of $G$ (the right hand side is a vector space direct limit). We equip $\mathcal{S}(G)$ with the direct limit topology. Let $C_c^\infty(G{\sslash}K)$ denote the subspace of  functions in $\mathcal{S}(G{\sslash}K)$ that are compactly supported. Then $C_c^\infty(G{\sslash}K)$ is a unital convolution algebra that acts on $L^{2}(G)$, again via convolution. If we denote by $C^*_r(G{\sslash}K)$ the $C^*$-algebra generated by $C_c^\infty(G{\sslash}K)$ inside $\mathcal{L}(L^2(G))$, then
$$C^*_r(G) = \varinjlim C^*_r(G{\sslash}K),$$
where the right hand side a direct limit of the C*-subalgebras $C^*_r(G{\sslash}K)$ partially ordered by inclusion. 

We single out some properties of $\mathcal{S}(G)$ that will be of importance to us. A unitary representation $\pi$ of $G$ is called {\em tempered} if for any smooth vectors $v,v' \in \pi$, there exists a constant $d$ such that we have 
$$| \langle \pi(g)(v),v' \rangle | \leq d\ \Xi(g)$$
for all $g \in G$ (see \cite[Eq. 2.2.3]{Beuzart-Plessis-20}). 

It is well-known that tempered representations of $G$ are precisely those that are weakly contained in the regular representation of  $G$ so that the tempered dual $\widehat{G}_{{\rm temp}}$ of $G$ is the same as the reduced dual $\widehat{G}_{{\rm red}}$ that we discussed earlier. However, it is standard to use the terminology tempered dual in the setting of reductive groups.

\begin{theorem} \label{HCS} The Schwartz algebra $\mathcal{S}(G)$ enjoys the following properties.
\begin{enumerate} 
\item If $\pi$ is a tempered representation and $\pi^\infty$ its associated smooth representation, then the $G$-action on $\pi^\infty$ integrates to an action of $\mathcal{S}(G)$.
\item If $K$ is a compact open subgroup of $G$, then $\mathcal{S}(G{\sslash}K)$ is a dense $*$-subalgebra of $C^*_r(G{\sslash}K)$, in particular the inclusion  $\mathcal{S}(G)\to C^{*}_{r}(G)$ is continuous and has dense range;
\item If an element of $\mathcal{S}(G{\sslash}K)$ is invertible in  $C^*_r(G{\sslash}K)$, then it is already invertible in $\mathcal{S}(G{\sslash}K)$;
\end{enumerate}
\end{theorem}
\begin{proof} The first claim is classical. It follows from the factorization $\mathcal{S}(G) = C_c^\infty(G) \star \mathcal{S}(G)$ (see \cite[(2.1.1) and (2.2.7)]{Beuzart-Plessis-20}). The other two claims are due to Vign\'eras \cite[Prop. 13]{Vigneras-90}\footnote{It was discovered in \cite[p. 9]{Solleveld} that Prop. 13 of \cite{Vigneras-90} appears to be false in general. However, it is true in the special case where the scale function $\sigma$ is such that $\sigma-1$ is a length function, which is precisely the case for us.}. An alternative proof is given in \cite[Lemma 2]{Brodzki-Plymen-02} for $G=GL_m(F)$. 
\end{proof}

\subsection{Metaplectic and orthogonal groups} The definition of the Schwartz algebra can be adapted to the non-linear group\footnote{While $Mp(W)$ is not linear like $Sp(W)$, it still is an ``$\ell$-group'' like $Sp(W)$: i.e. it is a Hausdorff topological group with a basis of neighborhoods of the identity consisting of compact open subgroups.} $\Mp(W)$   and the disconnected $O(V)$ in a straightforward way. For $\Mp(W)$, we pull-back the Harish-Chandra function $\Xi$ and the standard norm $\norm{{\cdot}}$ from $\Sp(W)$ to $\Mp(W)$. For $O(V)$, noting that $O(V)\simeq SO(V) \times \{\pm 1\}$, we extend the  $\Xi$ and $\norm{{\cdot}}$ from $SO(V)$ to $O(V)$ by declaring that $\Xi(-g)=\Xi(g)$ and $\norm{-g}=\norm{g}$. Then in both cases, $\Xi$ continues to enjoy the usual properties and $1+\log \norm{{\cdot}}$ still defines a length function\footnote{A {\em length function} on a group $G$ is a continuous function $L:G \to [0,\infty]$ such that $L(e)=1$, $L(g^{-1})=L(g)$ and $L(gh) \leq L(g)+L(h)$ for all $g,h \in G$.}. The definition of the Schwartz algebra now applies. A careful treatment for the case of the metaplectic group can be found in \cite[Section 2.3]{Li-12}. 

Thm. \ref{HCS} stays valid when $G$ is $\Mp(W)$ or  $O(V)$. The first part of Thm. \ref{HCS} is clear (see \cite[Section 2.3]{Li-12} for the metaplectic case). For the third part, the key properties that are needed for Vign\'eras' results to apply are  
\begin{itemize}
\item[(i)] given a compact open subgroup $K$ of $G$, the double coset space $K\backslash G/K$ has {\em polynomial growth} \cite[p. 237]{Vigneras-90} with respect to the scale $\sigma:=L-1=\log \norm{{\cdot}}$, (this is required for results in \cite[Section 6]{Vigneras-90})
\item[(ii)] $\Xi \sigma^{-r} \in L^2(G)$ for large enough $r>0$. (required in order to be able to apply \cite[Thm. 20]{Vigneras-90}, see \cite[Lem. 27 and Prop. 28]{Vigneras-90})
\end{itemize}
It is easy to see that these conditions are both satisfied for the cases of $G=\Mp(W)$ and $G=O(V)$. For example, let $K$ be a compact open subgroup of $\Mp(W)$. Then if $B'$ is the image of $K$ (again compact open) in $\Sp(W)$ under the covering map, then the natural map $K\backslash \Mp(W) /K \to K'\backslash \Sp(W)/K'$ has fibers of size at most 2, and it follows that polynomial growth of the latter implies the same for the former. This addresses item (i). For item (ii), we simply observe that the integral over $\Mp(W)$ is twice that over $\Sp(W)$. Similar  reasoning applies to the case of $O(V)$.


\section{Strong Morita equivalence}
\label{Morita}
In this section, we give an exposition of the notion of strong Morita equivalence for $C^{*}$-algebras. A Morita equivalence between $C^{*}$-algebras $A$ and $B$ induces a bijection between their Hilbert space representations. Good references for the $C^{*}$-theory include \cite{Raeburn-Williams, Lance}.

Using the $C^{*}$-algebras $C^{*}_{r}(H)$ and $C^{*}_{r}(G)$, we will exploit this in the context of local theta correspondence. However, as the matrix coefficients of the oscillator representation live in the Schwartz algebras $\mathcal{S}(G)$ and $\mathcal{S}(H)$ in the equal rank set-up, our constructions naturally start at the level of these algebras.

\subsection{Local subalgebras of $C^{*}$-algebras} 
Let $\mathcal{A}$ be a complex $*$-algebra. The {\em spectrum} of $a\in\mathcal{A}$ is the set 
\[\sigma_{\mathcal{A}}(a):=\{\lambda\in\mathbb{C}: a-\lambda\,\,\textnormal{is not invertible in }\mathcal{A}\}.\]
If $\mathcal{A}$ is nonunital, its \emph{unitisation} is the space
\[\mathcal{A}^{+}:=\mathcal{A}\oplus\mathbb{C},\]
equipped with coordinatewise addition and multiplication
\[(a,\lambda)\cdot (b,\mu):=(ab+\lambda b + \mu a,\lambda\mu).\] 
For nonunital algebras, we define $\sigma_{\mathcal{A}}(a):=\sigma_{\mathcal{A}^{+}}((a,0))$. An element $a\in\mathcal{A}$ is \emph{positive} if $a=a^{*}$ and $\sigma_{\mathcal{A}}(a)\subset \mathbb{R}_{\geq 0}$.

If $\mathcal{A}_0\subset \mathcal{A}$ is a  $*$-subalgebra 
we say that $\mathcal{A}_0$ is {\bf spectral invariant} in $\mathcal{A}$ if for all $a\in\mathcal{A}_0$ we have
\[\sigma_{\mathcal{A}_0}(a)=\sigma_{\mathcal{A}}(a).\]

A Fr\'echet $*$-algebra admits {\bf holomorphic functional calculus} (\cite[Lemma 1.3]{Phillips-91}): if $a\in\mathcal{A}$, $U\subset \mathbb{C}$ an open set containing $\sigma_{\mathcal{A}}(a)$ and $f:U\to \mathbb{C}$ a holomorphic function, then we can define an element $f(a)\in\mathcal{A}$ via 
\[f(a):=\int_{C}f(\lambda)(a-\lambda)^{-1}\mathrm{d}\lambda\in\mathcal{A},\]
where $C$ is a simple closed curve in $U$ enclosing $\sigma(a)$. 
Now suppose that $A$ is a $C^{*}$-algebra and $\mathcal{A}\subset A$ a $*$-subalgebra. We say that $\mathcal{A}$ is {\bf stable under holomorphic functional calculus} in $A$ if for all $a\in\mathcal{A}$ and $f$ a holomorphic function on neighbourhood of $\sigma_{\mathcal{A}}(a)$, we have $f(a)\in\mathcal{A}$. 
\begin{definition}[cf. \cite{Schweitzer-92}] Let $A$ be a $C^{*}$-algebra and $\mathcal{A}\subset A$ a dense $*$-subalgebra. We say that $\mathcal{A}$ is {\bf local} in $A$ if $\mathcal{A}$ is stable under holomorphic functional calculus, and {\bf spectral invariant} if for all $a\in\mathcal{A}$ we have $\sigma_{\mathcal{A}^{+}}(a)=\sigma_{A^{+}}(a)$.
\end{definition}
\begin{lemma}[cf. Lemma 1.2 in \cite{Schweitzer-92}] \label{spec-inv} Suppose that $\mathcal{A}$ is a Fr\'echet $*$-algebra, $A$ a $C^{*}$-algebra and $i:\mathcal{A}\to A$ a continuous injective $*$-homomorphism with dense range. Then $i(\mathcal{A})\subset A$ is local if and only if $i(\mathcal{A})$ is spectral invariant in $A$.
\end{lemma}
In the above situation we identify $\mathcal{A}$ with its image $i(\mathcal{A})$ and simply say that $\mathcal{A}\subset A$ is local.

\begin{proposition} \label{locality} Let $F$ be a non-archimedean local field of characteristic $0$. Let $G$ be either the $F$-points of a connected reductive group, or be $Mp(W)$ or $O(W)$ where $W$ is over $F$. Then  $\mathcal{S}(G)\subset C^{*}_{r}(G)$ is local.
\end{proposition}
\begin{proof} Since $\mathcal{S}(G)\subset C^{*}_{r}(G)$ is dense, it remains to show spectral invariance. Since $\mathcal{S}(G)\subset C^{*}_{r}(G)$ and as both algebras are nonunital, we have
\[\sigma_{C^{*}_{r}(G)}(a)\subset \sigma_{\mathcal{S}(G)}(a).\]
We have argued in Section \ref{HCSA} that for any compact open subgroup $K$ of $G$, if an element of $\mathcal{S}(G{\sslash}K)$ is invertible in  $C^*_r(G{\sslash}K)$, then it is already invertible in $\mathcal{S}(G{\sslash}K)$ (see Theorem \ref{HCS}(iii)). It follows from the definitions then that $\mathcal{S}(G{\sslash}K)$ is local in $C^*_r(G{\sslash}K)$ and hence spectral invariant by Lemma \ref{spec-inv}. 

For a $C^{*}$-subalgebras $B\subset A$ with $B$ unital and $A$ nonunital, we have for $b\in B$ that $\sigma_{A}(b)=\sigma_{A}(b)\cup\{0\}=\sigma_{B}(b)\cup\{0\}$ (see \cite[II.6.7]{BlackadarOA}). Now let $a\in\mathcal{S}(G)=\bigcup_{K} \mathcal{S}(G{\sslash}K)$, so $a\in \mathcal{S}(G{\sslash}K)$ for some $K$. Since $C^{*}_{r}(G)$ and $\mathcal{S}(G)$ are nonunital, whereas $\mathcal{S}(G{\sslash}K)$ and $C^*_r(G{\sslash}K)$ are unital, we find 
\begin{align*}
\sigma_{\mathcal{S}(G)}(a)\subset \sigma_{\mathcal{S}(G//K)}(a)\cup\{0\}=\sigma_{C^{*}_{r}(G//K)}(a)\cup\{0\}=\sigma_{C^{*}_{r}(G)}(a).
\end{align*}
We deduce that $\sigma_{C^{*}_{r}(G)}(a)= \sigma_{\mathcal{S}(G)}(a)$ as desired.
Now if $f$ is holomorphic on $\sigma(a)$ and $a\in \mathcal{S}(G{\sslash}K)$ then $f(a)\in \mathcal{S}(G{\sslash}K)\subset \mathcal{S}(G)$, proving that $\mathcal{S}(G)\subset C^{*}_{r}(G)$ is local.
\end{proof}

\subsection{Inner product modules} \label{hilbert-modules}

Let $B$ be a $C^{*}$-algebra and $\mathcal{B}\subset B$ a $*$-subalgebra.
A (complex) vector space $\mathcal{X}$ is called a right  {\bf inner product $\mathcal{B}$-module} if $\mathcal{X}$ is a right $\mathcal{B}$-module and it is equipped with a 
$\mathcal{B}$-valued positive-definite Hermitian form that is compatible with the right $\mathcal{X}$-module structure. More precisely, there is a sesquilinear map 
$$\langle \cdot, \cdot \rangle\subrangle{\mathcal{B}}: \mathcal{X} \times \mathcal{X} \to \mathcal{B},$$
satisfying the following properties:
\begin{enumerate} 
\item $\langle x,y \rangle\subrangle{\mathcal{B}}^*=\langle y, x \rangle\subrangle{\mathcal{B}}$ for all $x,y \in \mathcal{X}$,
\item $\langle x, y b  \rangle\subrangle{\mathcal{B}} = \langle x,y \rangle\subrangle{\mathcal{B}} b$ for all $x,y \in \mathcal{X}$ and $b \in \mathcal{B}$,
\item $\langle x,x \rangle\subrangle{\mathcal{B}}$ is a positive element of $B$ for every $x \in \mathcal{X}$.
\end{enumerate}
Observe that when $B=\C$, the Hilbert module $\mathcal{X}$ is simply a Hilbert space\footnote{All Hilbert spaces in this paper will be right Hilbert spaces.}.

The span of the set $\{ \langle x,y \rangle\subrangle{\mathcal{B}} \mid x,y \in \mathcal{X} \}$ is an ideal of $\mathcal{B}$. We call $\mathcal{X}$ {\bf full} if this ideal is dense in the ambient $C^{*}$-algebra $B$. The inner product module $\mathcal{X}$ is {\bf nondegenerate} if 
$$\langle x,x \rangle\subrangle{\mathcal{B}}=0\Leftrightarrow x=0.$$ 
We define left inner product modules in a similar way using left linear inner products. 

If $\mathcal{X}$ is an inner product $\mathcal{B}$-module, then 
\begin{equation}\label{module norm}
\norm{x}^2:=\norm{\langle x,x \rangle\subrangle{\mathcal{B}}}_B,
\end{equation}
defines a norm on $\mathcal{X}$. 

In the above, making the particular choice $\mathcal{B}=B$, we arrive at the following definition.
\begin{definition} Let $B$ be a $C^{*}$-algebra. An inner product $B$-module $X$ is a {\bf Hilbert $C^{*}$-module} if $X$ is complete with respect to the norm \eqref{module norm}.
\end{definition}
If $\mathcal{B}\subset B$ is a dense $*$-subalgebra and $\mathcal{X}$ a nondegenerate  
inner product $\mathcal{B}$-module, then the completion $X$ of $\mathcal{X}$ in the norm \eqref{module norm} is a Hilbert $C^{*}$-module over $B$ (\cite[Lemma 2.16]{Raeburn-Williams}). In the sequel we will construct inner product modules over the Schwartz algebra $\mathcal{S}(G)$ of a topological group $G$. Since  $\mathcal{S}(G)\subset C^{*}_{r}(G)$ is dense, such modules admit a completion as Hilbert $C^{*}$-modules over $C^{*}_{r}(G)$.

\begin{definition} Let $\mathcal{A}\subset A$ and $\mathcal{B}\subset B$ be dense $*$-subalgebras and $\mathcal{X}$ a right inner product $\mathcal{B}$-module. We say that $\mathcal{X}$ is an $(\mathcal{A},\mathcal{B})${\bf -correspondence} if $\mathcal{X}$ is a left $\mathcal{A}$-module such that
$$\langle ax,y\rangle_{\mathcal{B}}=\langle x, a^{*}y\rangle_{\mathcal{B}},\quad \forall x,y\in\mathcal{X},\, a\in\mathcal{A}.$$
In case $\mathcal{A}=A$, $\mathcal{B}=B$ and $X$ is an $(A,B)$-correspondence that is a Hilbert $C^{*}$-module over $B$, we say that $X$ is a {\bf $C^{*}$-correspondence} for $(A,B)$.
\end{definition}
Given a Hilbert $C^{*}$-module over a $C^{*}$-algebra $B$, its algebra of adjointable operators is the space
\[\textnormal{End}^{*}(X):=\left\{T: X\to X:\exists \ T^{*}:X\to X\,\,\forall x,y\in X\,\,\langle Tx,y\rangle=\langle x, T^{*}y\rangle\right\}.\]
Elements of $\textnormal{End}^{*}(X)$ are automatically bounded and right $B$-linear (that is $T$ is linear and $T(xb)=T(x)b$ for all $x \in X$ and $b \in B$). In fact $\textnormal{End}^{*}(X)$ forms a $C^{*}$-algebra in the operator norm it derives from the norm on $X$. Thus, for a $C^{*}$-correspondence $X$, we in fact have a $*$-homomorphism $A\to \textnormal{End}^{*}(X)$ between $C^{*}$-algebras.

Let $\mathcal{X}$ be an $(\mathcal{A},\mathcal{B})$-correspondence for dense subalgebras $\mathcal{A}\subset A$ and $\mathcal{B}\subset B$. We have discussed above that $\mathcal{X}$ can be completed to a right Hilbert $B$-module $X$. We will now see that if $\mathcal{A}$ is a {\em local} subalgebra of $A$, then $X$ can be promoted to a $C^{*}$-correspondence for $(A,B)$.

\begin{proposition} \label{premod} Let $\mathcal{A}\subset A$ and $\mathcal{B}\subset B$ be dense subalgebras and $\mathcal{X}$ an $(\mathcal{A},\mathcal{B})$-correspondence. If $\mathcal{A}$ is local, then for all $x\in\mathcal{X}$ and $a\in\mathcal{A}$ the inequality
\begin{equation} \label{continuity} \langle ax,ax\rangle\subrangle{\mathcal{B}}\leq \|a\|^{2}_{A}\langle x,x\rangle\subrangle{\mathcal{B}},
\end{equation}
holds true in the $C^{*}$-algebra $B$. 
\end{proposition} 
\begin{proof}
Let $\varepsilon>0$ and for $a\in\mathcal{A}$ and $x\in\mathcal{X}$ consider
\begin{align}\label{epdiff}
(\|a\|^{2}_{A}+\varepsilon)\langle x,x\rangle\subrangle{\mathcal{B}}-\langle a\cdot x, a \cdot x\rangle\subrangle{\mathcal{B}} &= \langle (\|a\|^{2}_{A}+\varepsilon-a^{*}a)x,x\rangle_{\mathcal{B}} 
\end{align}
and observe that by spectral invariance the element $\|a\|^{2}_{A}+\varepsilon-a^{*}a\in \mathcal{A}^{+}$ is positive invertible, and thus has spectrum contained in $[\varepsilon,M]$ for some $M>0$. Since the square-root function on $[\varepsilon,M]$ extends uniquely to a holomorphic function on an open neighbourhood of $[\varepsilon,M]$, 
we have that $(\|a\|^{2}_{A}+\varepsilon-a^{*}a)^{1/2}\in \mathcal{A}^{+}$, and thus that $(\|a\|^{2}_{A}+\varepsilon-a^{*}a)^{1/2}x\in\mathcal{X}$. We find that 
\begin{align*}
 \langle (\|a\|^{2}_{A}+\varepsilon-a^{*}a)x,x\rangle_{\mathcal{B}}  &=  \langle (\|a\|^{2}_{A}+\varepsilon-a^{*}a)^{1/2}x,(\|a\|^{2}_{A}+\varepsilon-a^{*}a)^{1/2}x\rangle_{\mathcal{B}} \geq 0,
\end{align*}
and therefore by \eqref{epdiff}
\begin{equation}\label{allep}
\langle ax, ax\rangle\subrangle{\mathcal{B}} \leq \|a\|^{2}_{A}\langle x,x\rangle\subrangle{\mathcal{B}} + \varepsilon\langle x,x\rangle\subrangle{\mathcal{B}}.
\end{equation}
Since \eqref{allep} holds for all $\varepsilon>0$, we conclude that (\ref{continuity}) holds as desired.
\end{proof}

\begin{remark} Inequality (\ref{continuity}) is one of the defining properties of a \emph{pre-imprimitivity bimodule}, see \cite[Definition 3.9]{Raeburn-Williams}.
\end{remark}


\begin{remark} In the proof of Prop. \ref{premod}, we have used the locality of the Schwartz algebra $\mathcal{S}(G)$ inside $C^{*}_{r}(G)$ (see Prop. \ref{locality}). We will use this result again later for our second application in Section \ref{App-2} where we use the fact that locality implies that the inclusion of $\mathcal{S}(G)$ into  $C^{*}_{r}(G)$ induces an isomorphism in $K$-theory.
\end{remark}

\begin{corollary}
\label{completion} 
Let $\mathcal{A}\subset A$ and $\mathcal{B}\subset B$ be dense $*$-subalgebras, $\mathcal{X}$ an $(\mathcal{A},\mathcal{B})$-correspondence and 
\[\mathcal{N}:=\left\{x\in\mathcal{X}:\langle x,x\rangle\subrangle{\mathcal{B}}=0\right\},\]
the radical of $\mathcal{X}$. Assume that $\mathcal{A}$ is local. Then $\mathcal{X}/\mathcal{N}$ can be completed into a $C^{*}$-correspondence for $(A,B)$.
\end{corollary}
\begin{proof} By Proposition \ref{premod}, $\mathcal{A}$ maps $\mathcal{N}$ into itself, so $\mathcal{X}/\mathcal{N}$ becomes a nondegenerate $(\mathcal{A},\mathcal{B})$- correspondence. Let $X$ denote its completion as a right Hilbert $C^{*}$-module over $B$. Then applying Proposition \ref{premod} once more, we deduce that the action of $\mathcal{A}$ on $\mathcal{X}/\mathcal{N}$ is bounded with respect to the $C^{*}$-module norm on $\mathcal{X}/\mathcal{N}$, and thus extends to an action of $\mathcal{A}$ on $X$ by adjointable operators. Since $\mathcal{A}\subset A$ is dense in the $C^{*}$-norm on $A$, another application of Proposition \ref{premod} shows that the left $\mathcal{A}$-module structure extends to a left $A$-module structure satisfying $\langle ax,y\rangle=\langle x,a^{*}y\rangle$ for all $a\in A$ and $x,y\in X$. Hence $X$ is a $C^{*}$-correspondence for $(A,B)$.
\end{proof}

\subsection{Induction of representations of $C^{*}$-algebras} \label{gen-ind}  
For $C^{*}$-correspondences over a pair of $C^{*}$-algebras $(A,B)$ there is a far reaching theory of induced representations. Due to the Gelfand-Naimark-Segal theorem, a $C^{*}$-algebra usually admits numerous Hilbert space representations. A $C^{*}$-correspondence for $(A,B)$ allows one to construct a representation of $A$ given a representation of $B$. In the case of group $C^{*}$-algebras, this gives a functorial correspondence between the associated group representations.

\subsubsection{The interior tensor product}Let $A$ and $B$ be $C^{*}$-algebras and $X$ a $C^{*}$-correspondence for $(A,B)$. Then we can ``induce'' representations of $B$ to $A$ via $X$ via the following tensor product construction.  

\begin{proposition}[{\cite[Proposition 4.5]{Lance}}]\label{internal-tensor-product}
Let $X$ be a $C^{*}$-correspondence for $(A,B)$ and $\pi : B \to \mathcal{L}(V_{\pi})$ a representation of $B$ on a Hilbert space $V_{\pi}$. Consider  $X \otimes^{\mathrm{alg}} V_{\pi}$ the algebraic tensor product of vector spaces. The right sesquilinear form 
\begin{equation} \label{localisation-inner-product} (x \otimes v, x' \otimes v') := \langle v, \pi(\langle x,x' \rangle\subrangle{B})v' \rangle\subrangle{V_\pi}
\end{equation}
is positive and its radical 
$$N_{\pi}:=\left\{\xi \in X \otimes^{\mathrm{alg}} V_{\pi} \mid (\xi, \xi) = 0 \right\}$$
is equal to the \emph{balancing subspace} spanned by elements of the form
$$xb \otimes v - x\otimes \pi(b)(v) \qquad (x \in X, v \in V_{\pi}, b \in B).$$
The completion of $(X \otimes^{\mathrm{alg}} V_{\pi})/N_\pi$ with respect to the inner product \eqref{localisation-inner-product} is a Hilbert space that we denote by $X \otimes_{B}V_{\pi}$ and is commonly called is the {\bf internal tensor product of $X$ and $V_{\pi}$ over $B$}. 
\end{proposition}
\begin{proof}
The proof of positivity is based on the fact that for $x_{1},\cdots, x_{n}\in X$ the matrix $\varepsilon:=\pi(\langle x_{i},x_{j})_{B})_{ij}\in M_{n}(\mathcal{L}(V))$ is a positive operator on $V^{n}$, so that for $\xi=\sum_{i=1}^{n}x_{i}\otimes v_{i}$ we have
\begin{align*}
(\xi,\xi)=(v, \varepsilon\cdot  v)\geq 0,\quad v:=\begin{pmatrix} v_{1} \\ \vdots \\ v_{n}\end{pmatrix}\in V^{n}.
\end{align*}
The fact that the radical coincides with the balancing subspace uses positivity of the matrix $\varepsilon$ and the fact the $M_{n}(B)$ is again a $C^{*}$-algebra, so we can extract square roots.
\end{proof}

It can then be shown that the action 
$$a(x \otimes v) := ax \otimes v$$ 
of $A$ on the space $X \otimes^{\mathrm{alg}} V_{\pi}$ gives rise to representation of $A$ on the Hilbert space $X \otimes_B V_{\pi}$ which we will denote 
$$\textnormal{\small Ind}_{B}^{A}(X, \pi),$$
and refer to as the $A$-representation {\bf induced from $\pi$ via $X$}.

\subsubsection{Functoriality of induction} \label{sec: functorilaity} The induction procedure that we described above is functorial (see \cite[Prop. 2.69]{Raeburn-Williams}). 
\begin{proposition} Let $A,B$ be two $C^*$-algebras and let $X$ be a $C^*$-correspondence for $(A,B)$. Assume that the action of $A$ on $X$ is non-degenerate (i.e. $A{\cdot} X =X$). Then the map 
$$\pi \mapsto \textnormal{\small Ind}_{B}^{A}(X, \pi)$$
is a functor from the category of non-degenerate representations of $B$ with bounded intertwining operators to the corresponding category of $A$,  
which at the level of morphisms takes the form $ T  \mapsto {\bf 1}\otimes T$.
\end{proposition}

It follows that induction respects unitary equivalence and direct sums.

\subsubsection{Continuity of induction}  \label{sec: continuity}
Let $C$ be a $C^*$-algebra. Given a $*$-representation $\pi$ and a set of representations $S$ of $C$,  we say that $\pi$ is {\bf weakly contained} in $S$ (denoted
$\pi  \prec S$) if
\begin{equation} \label{weak-containment} {\rm ker (\pi)} \supset \bigcap_{\sigma \in S} {\rm ker} (\sigma).
\end{equation}

Let ${\rm Rep}(C)$ denote the collection\footnote{In order to ensure that the collection is a set, we actually fix a cardinal $\aleph$ and consider $*$-representations on Hilbert spaces of cardinality $\leq \aleph$. For us, considering separable Hilbert spaces will suffice.} of equivalence classes of all $*$-representations of $C$. This space comes equipped with a second-countable topology, due to Fell, that is characterized as follows: a net $\{ T_i \}$ of elements of ${\rm Rep}(C)$ converges to $T$ if and only if $T$ is weakly contained in every subnet of $\{ T_i \}$.
 
Let $\widehat{C}$ denote the subset of irreducible elements of ${\rm Rep}(C)$. We call $\widehat{C}$ the {\em spectrum of $C$}. Relativized to $\widehat{C}$, closure in Fell topology agrees with weak closure: a subset $S \subset \widehat{C}$ is closed if and only if $S= \{ \pi \in \widehat{C} \mid \pi \prec S \}$. In this case, the topology agrees with the pull-back of the hull-kernel topology on the space of primitive ideals of $C$. See \cite{Fell-64} for a discussion of these topologies. 

We have the following continuity result (see \cite[Prop. 2.72]{Raeburn-Williams})
\begin{proposition} \label{prop: continuity} Let $A,B$ be two $C^*$-algebras and let $X$ be a $C^*$-correspondence for $(A,B)$. Then the map 
$$\pi \mapsto \textnormal{\small Ind}_{B}^{A}(X, \pi)$$
is continuous with respect to the Fell topologies on ${\rm Rep}(B)$ and on ${\rm Rep}(A)$.
\end{proposition}

\subsection{Equivalence bimodules}
\begin{definition} \label{equivalence-bimodule} Let $\mathcal{A}\subset A$ and $\mathcal{B}\subset B$ be local subalgebras and $\mathcal{X}$ an algebraic $(\mathcal{A},\mathcal{B})$-bimodule in the usual sense that the two actions commute. We call $\mathcal{X}$ an {\bf inner product bimodule} if 
\begin{enumerate}
\item $\mathcal{X}$ is a left inner product $\mathcal{A}$-module and a right inner product $\mathcal{B}$-module,
\item for all $x,y \in \mathcal{X}$, $a \in \mathcal{A}$ and $b \in \mathcal{B}$, we have
\begin{equation} \langle a {\cdot} x,y \rangle\subrangle{\mathcal{B}}=\langle x, a^*{\cdot} y \rangle\subrangle{\mathcal{B}} \qquad \prescript{}{\mathcal{A}}{\langle x{\cdot} b,y \rangle}=\prescript{}{\mathcal{A}}{\langle x, y{\cdot} b^* \rangle},
\end{equation}
\item  for all $x,y,z \in \mathcal{X}$, we have 
\begin{equation} \prescript{}{\mathcal{A}}{\langle x, y \rangle} {\cdot} z = x {\cdot} \langle y,z \rangle\subrangle{\mathcal{B}}.
\end{equation}
\end{enumerate}
\end{definition}

\subsubsection{Strong Morita equivalence of $C^{*}$-algebras} 
Given an inner product bimodule $\mathcal{X}$ for local subalgebras $\mathcal{A}\subset A$ and $\mathcal{B}\subset B$, we obtain a norm on $\mathcal{X}$ for each of the inner products. The following is well-known (see e.g. \cite[Prop. 3.11]{Raeburn-Williams}). We give the short proof here for the convenience of the reader.

\begin{proposition} Let $\mathcal{A}\subset A$ and $\mathcal{B}\subset B$ be local subalgebras and $\mathcal{X}$ an $(\mathcal{A},\mathcal{B})$ inner product bimodule. Then the norms
\begin{equation}
\label{equalnorms}
\|x\|_{A}:=\|_{\mathcal{A}}\langle x,x\rangle\|_{A}^{1/2},\quad \|x\|_{B}:=\|\langle x,x\rangle_{\mathcal{B}}\|_{B}^{1/2},
\end{equation}
on $\mathcal{X}$ are equal. 
\end{proposition} 
\begin{proof} Using Definition \ref{equivalence-bimodule} and the Cauchy-Schwartz inequality for Hilbert $C^{*}$-modules, we have
\begin{align*} \norm{\langle x, x \rangle\subrangle{\mathcal{B}}}^2_{B} &= \norm{\langle x, x \rangle\subrangle{\mathcal{B}}\langle x, x \rangle\subrangle{\mathcal{B}}}_{B}
=\norm{\langle x, x{\cdot} \langle x, x \rangle\subrangle{\mathcal{B}} \rangle\subrangle{\mathcal{B}}}_{B}
=\norm{\langle x, \prescript{}{\mathcal{A}}{\langle x, x \rangle}{\cdot} x \rangle\subrangle{\mathcal{B}}}_{B} \\
&\leq 
\norm{\langle x, x \rangle\subrangle{\mathcal{B}}}^{1/2}_{B} \norm{\langle \prescript{}{\mathcal{A}}{\langle x, x \rangle}{\cdot}x, \prescript{}{\mathcal{A}}{\langle x, x \rangle}{\cdot} x \rangle\subrangle{\mathcal{B}}}^{1/2}_{B},
\end{align*}
for all $x\in\mathcal{X}$. Applying Proposition \ref{premod} to the second factor in the last term, we conclude that 
$$\norm{\langle x, x \rangle\subrangle{\mathcal{B}}}^2_{B} \leq 
\norm{\langle x, x \rangle\subrangle{\mathcal{B}}}^{1/2}_{B}  \norm{\prescript{}{\mathcal{A}}{\langle x, x \rangle}}_{A}
\norm{\langle x, x \rangle\subrangle{\mathcal{B}}}^{1/2}_{B}.$$
Cancelling a factor of $\norm{\langle x, x \rangle\subrangle{\mathcal{B}}}_{B}$ gives us $\norm{\langle x, x \rangle\subrangle{\mathcal{B}}}_{B} \leq \norm{\prescript{}{\mathcal{A}}{\langle x, x \rangle}}_{A}$. Swapping the roles of $A$ and $B$, we obtain the opposite inequality and obtain the desired equality of norms.
\end{proof}
\begin{definition}
Let $A$ and $B$ be $C^{*}$-algebras. A {\bf Hilbert $C^{*}$-bimodule} for $(A,B)$ is an inner product bimodule that is complete in the norms
\[\|x\|_{A}^{2}=\|\prescript{}{A}{\langle x,x\rangle}\|_{A}=\|x\|^{2}_{B}=\|\langle x,x\rangle\subrangle{B} \|_{B}.\]
A Hilbert $C^{*}$-bimodule for $(A,B)$ is a {\bf (Morita) equivalence bimodule} if both inner products are full. 
Two $C^*$-algebras $A$ and $B$ are called {\bf strongly Morita equivalent} if there exists an $(A,B)$ Morita equivalence bimodule $X$. 
\end{definition}
While for unital $C^*$-algebras $A$ and $B$, this notion coincides with classical Morita equivalence as rings, in general, it is stronger than the classical notion (hence the name).

\subsection{} 
We have the notion of ``dual'' of an $(\mathcal{A},\mathcal{B})$-inner product bimodule.  This is a $(\mathcal{B},\mathcal{A})$-inner product bimodule which is defined as follows.
\begin{definition} \label{dual-bimodule} Let $\mathcal{X}$ be an $(\mathcal{A},\mathcal{B})$-inner product bimodule and $\mathcal{X}^{*}$ its conjugate vector space. By definition, we have an anti-linear bijection $\beta: \mathcal{X} \to \mathcal{X}^{*}$ such that $\beta(\lambda{\cdot}x)= \overline{\lambda}{\cdot}\beta(x)$ for every $x \in \mathcal{X}$ and $\lambda\in\mathbb{C}$ a complex scalar. The {\bf dual module of $\mathcal{X}$} is $\mathcal{X}^{*}$ equipped with the following $(\mathcal{B},\mathcal{A})$-inner product bimodule structure
$$b{\cdot} \beta(x) :=\beta(xb^*), \qquad \prescript{}{\mathcal{B}}{\langle \beta(x), \beta(y)\rangle}:=\langle x, y \rangle\subrangle{\mathcal{B}},$$ 
$$\beta(x){\cdot}a :=\beta(a^*x), \qquad \langle \beta(x), \beta(y) \rangle\subrangle{\mathcal{A}}:=\prescript{}{\mathcal{A}}{\langle x, y \rangle}.$$
\end{definition}

\subsection{} \label{sec: homeomorphic} If $X$ is an $(A,B)$-equivalence bimodule, then the dual module $X^{*}$ is a $(B,A)$-equivalence bimodule. In fact there are isomorphisms of interior tensor products
\[X\otimes_{B}X^{*}\simeq A,\quad X^{*}\otimes_{A} X \simeq B,\]
as $(A,A)$ and $(B,B)$ $C^{*}$-bimodules, respectively.

It follows that the induction functors associated to $X$ and to $X^{*}$ are inverses to each other and therefore, by Prop. \ref{prop: continuity}, they implement a homeomorphism between the spectra of $A$ and $B$. In particular,  the categories of representations of $A$ and $B$ are equivalent. 

\subsection{}\label{ideals}
Given a two-sided closed ideal $J$ of $B$, the space $X_J$ given by the closure of the linear span of all $x {\cdot} b$ with $x \in X$ and $b \in J$ forms an $(I,J)$-equivalence bimodule where $I$ is the two-sided closed ideal of $A$ given by the closure of the linear span of all $\prescript{}{A}{\langle x {\cdot} b, y \rangle}$ with $x,y \in X$ and $b \in J$. We will view this association as an {\em induction of ideals} implemented by $X$ and accordingly denote $I$ by 
$$\textnormal{\small Ind}_B^A(X, J).$$
It is well-known (see \cite[Props. 3.24 and 3.25]{Raeburn-Williams}) that the map $J \mapsto \textnormal{\small Ind}_B^A(X, J)$ sets up a bijection between the two-sided closed ideals of $B$ and $A$ which respects inclusion of ideals (thus, it identifies the lattices of ideals of $B$ and $A$). Moreover, the induction of ideals is compatible with the induction of representations, that is, if $\pi$ is a representation of $B$, then 
$$\textnormal{\small Ind}_B^A(X, {\rm ker}(\pi)) = {\rm ker}(\textnormal{\small Ind}_B^A(X, \pi)).$$

\subsection{} Lastly, we observe that any Hilbert $C^{*}$-bimodule induces a Morita equivalence between certain associated $C^{*}$-algebras. Suppose that $X$ is a Hilbert $C^{*}$-bimodule for $(A,B)$. The sets
\begin{align*}
I_{X}&=\,_{A}\langle X,X\rangle:=\overline{\textnormal{span}\left\{_{A} \langle x,y\rangle:x,y\in X\right\}}\subset A,\\
J_{X}&=\langle X,X\rangle_{B}:=\overline{\textnormal{span}\left\{\langle x,y\rangle_{B}:x,y\in X\right\}}\subset B,
\end{align*}
form closed two-sided ideals (in particular, $C^*$-subalgebras) of $A$ and $B$ respectively. Since $X$ is an $(A,B)$ Hilbert $C^{*}$-bimodule and the inner products in fact take their values in the ideals $I_{X}\subset A$ and $J_{X}\subset B$, we can view $X$ as Hilbert $C^{*}$-bimodule over $(I_{X},J_{X})$. By construction, the inner products are now full so $X$ is a Morita equivalence bimodule for $(I_{X},J_{X})$. We now summarise our findings for future reference.

\begin{proposition}\label{constructMorita} Let $A$ and $B$ be $C^{*}$-algebras and $\mathcal{A}\subset A$ and $\mathcal{B}\subset B$ local subalgebras. Suppose that $\mathcal{X}$ is a nondegenerate $(\mathcal{A},\mathcal{B})$ inner product bimodule and denote by $X$ the completion of $\mathcal{X}$ in the norm \eqref{equalnorms}. Then $X$ is an $(A,B)$ Hilbert $C^{*}$-bimodule and hence a Morita equivalence bimodule for the pair of ideals $(I_{X},J_{X})$.
\end{proposition}


\section{The oscillator bimodule}
We consider the smooth oscillator representation $\omega$ of $G\times H$ realized on the space of smooth vectors that will denote by $\s.$ In this section we will equip $\s$ with the structure of an $(\mathcal{S}(G),\mathcal{S}(H))$ inner product bimodule in the sense of Definition \ref{equivalence-bimodule}. We will then use Proposition \ref{constructMorita} to complete $\s$ into a $C^{*}$-bimodule for $(C^{*}_{r}(G),C^{*}_{r}(H))$. 
\subsection{Matrix coefficients} 

Matrix coefficients of the oscillator representation are critical to our construction. So we start by recording a well-known fast decay property that our equal rank case enjoys.
\begin{proposition} \label{mat-coeffs} For all $x,y \in \s$, the matrix coefficient functions 
$$g \mapsto \langle  x, \omega(g)(y)\rangle, \qquad h \mapsto \langle x, \omega(h)(y) \rangle, $$ 
belong to the Schwartz algebras $\mathcal{S}(G),\mathcal{S}(H)$ of $G$ and $H$ respectively.
\end{proposition}
\begin{proof} This is well-known to the experts. It follows from the matrix coefficient estimates of Li (see Cor. 3.4 and proof of Thm 3.2 in \cite{Li-89}). See also Prop. 3.1.1 of \cite{Sakellaridis-17} (where $G_2$ denotes the ``smaller'' group of the dual pair and hence applies to both groups in our equal rank case). The same observation is made in  \cite[Lemma 7.4]{Ichino-22} for the case of real unitary groups.
\end{proof}
Observe that in particular, the oscillator representation, when restricted to $G$ or $H$, is tempered (see Section \ref{HCSA}). 

\subsection{Right inner product module structure} \label{right-module-structure} We equip $\s$ with a {\em right} $\mathcal{S}(H)$-module structure as follows:  for  $x \in \s$
\begin{equation} \label{right-module} x {\cdot} b := \int_H b(h) \omega(h^{-1})(x) \ \mathrm{d}h, \qquad b \in \mathcal{S}(H).
\end{equation}
Note that $x {\cdot} b$ is well-defined and belongs to $\s$ since $\omega$ is tempered as an $H$-representation (see Prop. \ref{mat-coeffs}). Next, we equip $\s$ with an $\mathcal{S}(H)$-valued sesquilinear map
\begin{equation} \label{right-inner-product} \langle x,y \rangle_H (h) := \langle x, \omega(h)(y) \rangle, \qquad x,y \in \s, \ h \in H
\end{equation}
The form $\langle {\cdot},{\cdot} \rangle_H$ is Hermitian: for all $h \in H$ and $x,y \in \s$, we have
$$\langle x,y \rangle\subrangle{H}^*(h)= \overline{\langle x,y \rangle\subrangle{H}(h^{-1})}=  \overline{\langle x, \omega(h^{-1})(y) \rangle}
= \overline{\langle \omega(h)(x), y \rangle}= \langle y, x \rangle\subrangle{H}(h).$$ 
It is routine to check that the form $\langle {\cdot},{\cdot} \rangle_H$ is compatible with the right $\mathcal{S}(H)$-module structure given above, that is,  
$$\langle x, y {\cdot} b  \rangle\subrangle{H} = \langle x,y \rangle\subrangle{H} b.$$

\begin{proposition} \label{construct-1} Equipped with the right module structure (\ref{right-module}) and the form (\ref{right-inner-product}), the space $\s$ becomes a nondegenerate right inner product module over $\mathcal{S}(H)$.
\end{proposition}
\begin{proof} We just need to prove that the form $\langle {\cdot}, {\cdot} \rangle\subrangle{H}$ is positive definite, that is, for any $x \in \s$, we have 
$\langle x,x \rangle\subrangle{H} \geq 0$ as an element of the $C^*$-algebra $C^{*}_{r}(H)$ and that $\langle x,x \rangle\subrangle{H} =0$ only when $x=0$. 

To show the former, it is enough to exhibit an injective representation $\Pi$ of $C^{*}_{r}(H)$ with the property that $\Pi(\langle \varphi, \varphi \rangle\subrangle{H})$ is a positive operator for every $\varphi \in \s$. If we prove that $\pi(\langle \varphi, \varphi \rangle\subrangle{H})$ is a positive operator for every $\pi$ in the spectrum of $C^{*}_{r}(H)$, then we will be done by considering the representation 
$$\Pi = \bigoplus_{\pi \in \widehat{C^*_{r}(H)}} \pi$$
which is injective since for every $a \in C^{*}_{r}(H)$, there is an irreducible representation $\pi$ such that $\norm{a}=\norm{\pi(a)}$ (see e.g. \cite[Thm. A.14]{Raeburn-Williams}). Therefore it suffices to prove that
$$\pi(\langle x,x \rangle\subrangle{H}) \geq 0$$
as an operator on $V_\pi$ for every $\pi$ in the spectrum of $C^{*}_{r}(H)$. 

Let $x, x' \in \s$ and consider the operator $\pi(\langle x,x' \rangle\subrangle{H})$ on $V_\pi$. This operator is determined by the bilinear form 
$$\bigl \langle v, \pi(\langle x,x' \rangle\subrangle{H})(v') \bigr \rangle$$
for $v,v' \in V_\pi$. We unfold the left hand side
\begin{align*} \left \langle v, \pi(\langle x, x' \rangle\subrangle{H})(v') \right \rangle &= 
\left \langle v, \int_{H} \langle x, x' \rangle\subrangle{H}(h) \pi_h(v') \mathrm{d}h \right \rangle \\
&=  \int_{H} \langle x, x' \rangle\subrangle{H}(h) \langle v, \pi_h(v') \rangle \mathrm{d}h \\
&=  \int_{H} \langle x, \omega(h)(x') \rangle \langle v, \pi_h(v') \rangle \mathrm{d}h\\
&=(x{\otimes}v, x' {\otimes}v')_\pi
\end{align*}
where $({\cdot},{\cdot})_\pi$ is the Hermitian form on $\s \otimes V_\pi$ (see \ref{Li-form}). The latter is non-negative thanks to Prop. \ref{form-positive}. Therefore, we conclude that for $x\in\s$ and $v\in V_{\pi}$ we have
$$\bigl \langle v, \pi(\langle x,x\rangle_{H})v\bigr \rangle= \bigl (x{\otimes}v, x {\otimes}v \bigr )_\pi\geq 0$$
which implies that $\langle x,x\rangle_{H}\geq 0$ in $C^{*}_{r}(H)$. 

Finally, to show definiteness, suppose $\langle x,x\rangle_{H}=0$, so that 
\[\langle \omega(h)x,x\rangle=0,\quad \forall h\in H.\]
Then in particular, for $h=e$ we find that $\langle x,x\rangle=0$, so that $x=0$ in the oscillator representation. Since $\s$ injects into the oscillator representation, we conclude that $x=0$ in $\s$.\end{proof}

\subsection{Left inner product module structure} 
We will show that $\T$ can also be obtained by equipping $\s$ with a left $\mathcal{S}(G)$ inner product module structure. As before, we first equip $\s$ with a {\em left} $\mathcal{S}(G)$-module structure:  for  $x \in \s$
\begin{equation} \label{left-module} a {\cdot} x := \omega(a)(x) = \int_G a(g) \omega(g)(x) \ \mathrm{d}g, \qquad a \in \mathcal{S}(G).
\end{equation}
Notice that $a {\cdot} x $ is well-defined and belongs to $\s$ since $\omega$, as a $G$-representation, is tempered (as evidenced by Prop. \ref{mat-coeffs}). Next, we equip $\s$ with an $\mathcal{S}(G)$-valued left linear form 
\begin{equation} \label{left-inner-product} \prescript{}{G}{\langle x,y \rangle}(g) := \langle \omega(g)(y), x \rangle, \qquad x,y \in \s, \ g \in G
\end{equation}
It is straightforward to check that this $ \mathcal{S}(G)$-valued form is Hermitian and compatible with the left $ \mathcal{S}(G)$-module structure given above.

\begin{proposition} \label{construct-2} Equipped with the left module structure (\ref{left-module}) and the form (\ref{left-inner-product}), the space $\s$ becomes a nondegenerate left inner product module over $\mathcal{S}(G)$.
\end{proposition}
\begin{proof} We just need to prove that the form $\prescript{}{G}{\langle {\cdot}, {\cdot} \rangle}$ is positive definite. As discussed in the proof of Prop. \ref{construct-1}, it suffices to prove that 
$\pi(\prescript{}{G}{\langle x,x \rangle})$ is positive as an operator on $V_\pi$ for every $\pi$ in the spectrum of $C^{*}_{r}(G)$. 

Given $v \in V_\pi$, calculations as in said proof show that $\bigl \langle v, \pi(\prescript{}{G}{\langle x,x \rangle})(v) \bigr \rangle$ equals 
$\overline{(x\otimes v, x\otimes v)_{\overline{\pi}}}$ in the notation of Section \ref{Li-construction}. Positivity again follows from Prop. \ref{form-positive}.
\end{proof}

\subsection{Key compatibility property} In fact $\s$ is an inner product bimodule for $(\mathcal{S}(G),\mathcal{S}(H))$ in the sense of Definition \ref{equivalence-bimodule}. We will prove the following compatibility between the two inner products: for $x,y,z \in \s$
\begin{equation} \label{compatibility} \prescript{}{G}{\langle x,y\rangle} {\cdot} z = x {\cdot} \langle y,z \rangle\subrangle{H}
\end{equation}

A convenient reformulation is as follows, for $x,y,z,u \in \s$,
$$ \Bigl \langle \prescript{}{G}{\langle x,y\rangle} {\cdot} z, u \Bigr \rangle  = \left \langle x {\cdot} \langle y,z \rangle\subrangle{H}, u \right \rangle$$
Unfolding two sides, we obtain

\begin{align*} 
\Bigl \langle \prescript{}{G}{\langle x,y\rangle} {\cdot} z, u \Bigr \rangle  
= \left \langle \int_G \prescript{}{G}{\langle x,y \rangle}(g) \ \omega(g)(z) \mathrm{d}g, u \right \rangle 
= \int_G \langle x, \omega(g)(y) \rangle \overline{\langle u, \omega(g)(z) \rangle} \mathrm{d}g 
\end{align*}
and
\begin{align*} 
\left \langle x {\cdot} \langle y,z \rangle\subrangle{H}, u \right \rangle &= \left \langle \int_H \langle y,z \rangle\subrangle{H}(h) \ \omega(h^{-1})(x) \mathrm{d}h, u \right \rangle 
=  \int_H \overline{\langle y, \omega(h)(z) \rangle} \langle x, \omega(h)(u) \rangle \mathrm{d}h 
\end{align*}
Thus we arrive at the formulation
\begin{equation} \label{poisson} \int_G \langle x, \omega(g)(y) \rangle \overline{\langle u, \omega(g)(z) \rangle} \mathrm{d}g 
 = \int_H \langle x, \omega(h)(u) \rangle \overline{\langle y, \omega(h)(z) \rangle} \mathrm{d}h.
\end{equation}

We will now prove the above equality holds once we scale the Haar measures on $G$ and $H$ appropriately. 
 \begin{proposition} \label{poisson-proof} The Haar measures on $G$ and $H$ may be chosen in such a way that Equation (\ref{poisson}), and hence (\ref{compatibility}), hold for any $x,y,z \in \s$.
 \end{proposition}
 \begin{proof} Let $(\overline{\omega},\overline{\s})$ denote the complex conjugate representation of $(\omega,\s)$. Consider the maps 
 $$P_G, P_H : \s \otimes \overline{\s} \otimes \s \otimes \overline{\s} \rightarrow \C,$$
 given by
 $$P_G\bigl (x\otimes \bar{z} \otimes y \otimes \bar{u} \bigr ):= \int_G \langle x, \omega(g)(y) \rangle \overline{\langle u, \omega(g)(z) \rangle} \mathrm{d}g,$$
 $$P_H\bigl (x\otimes \bar{z} \otimes y \otimes \bar{u} \bigr ):=\int_H \langle x, \omega(h)(u) \rangle \overline{\langle y, \omega(h)(z) \rangle} \mathrm{d}h.$$
Straight-forward calculations show that
$$P_G, P_H \in {\rm Hom}_{H \times \left ( G \times G \right )} \bigl ( \left ( \omega \otimes \overline{\omega} \right ) \otimes \left ( \omega \otimes \overline{\omega} \right ), {\bf 1}  \bigr )$$
where we consider $\omega \otimes \overline{\omega}$ with the diagonal action of $H$ and with the natural action of $G \times G$.

Now let $\W=W+(-W)$ where $-W$ denotes the space $W$ with the form $-\langle{\cdot},{\cdot}\rangle_W$. We have an oscillator representation 
$\bm{\omega}$ of $H \times G(\W)$ which satisfies\footnote{Recall that $\chi_V$ is one of the two auxillary characters of $E^\times$ that we fixed at the very beginning to make sure that the oscillator representation can be pulled back to $G\times H$.} (see \cite[Section 4]{Gan-Ichino-14}))
$$\omega \otimes \left ( \overline{\omega} \otimes \chi_V \right ) \simeq \bm{\omega}$$ 
as $G{\times}G$-representations. Here we embed $G(W){\times}G(-W)$ in $G(\W)$ and identify $G(-W)=G(W)$. 
 
In Section 17 of \cite{Gan-Ichino-14}, Gan and Ichino introduce two forms
$$\mathcal{I}, \mathcal{E} \in {\rm Hom}_{H {\times} \left ( G {\times} G \right )}(\bm{\omega} \otimes \overline{\bm{\omega}} \otimes \bar{\chi}_V \otimes \chi_V, {\bf 1}).$$ 
{\em Note that the roles of $G$ and $H$ in their Section 17 have to be swapped, as we did in the previous display and below, in order to make it compatible with ours.}

Remarkably, our forms $P_G$ and $P_H$ are essentially equal to the forms $\mathcal{I}$ and $\mathcal{E}$ of Gan and Ichino. To see this, observe that for $x,y,z,u \in \s$, we have 
\begin{align*}\mathcal{I}(x \otimes \overline{z}, \ y \otimes \overline{u}) &= \int_{G} \langle \omega(g)(x),y \rangle \overline{\langle \omega(g)(z),u \rangle} \mathrm{d}g \\ 
&=\int_{G} \langle x,\omega(g^{-1})(y) \rangle \overline{\langle z,\omega(g^{-1})(u) \rangle} \mathrm{d}g \\
&=\int_{G} \langle x,\omega(g)(y) \rangle \overline{\langle z,\omega(g)(u) \rangle} \mathrm{d}g \\
&= P_G(x \otimes \overline{u} \otimes \ y \otimes \overline{z})
\end{align*}
(for the first equality, see the middle of page 593 of \cite{Gan-Ichino-14}, plug in $g=g'=e$ and swap $G$ with $H$). Moreover, we have (plugging in $g=g'=e$, and swapping $G$ with $H$, in line 7 of p 594 of \cite{Gan-Ichino-14}):
\begin{align*}\mathcal{E}(x \otimes \overline{z}, \ y \otimes \overline{u}) &= \int_H \mathcal{F}_{x\otimes \bar{z}}(i(h,1)) \overline{ \mathcal{F}_{y \otimes \bar{u}}(i(h,1))} \mathrm{d}h
\end{align*}
where by definition (see top of page 587 of \cite{Gan-Ichino-14}) we have 
$$\mathcal{F}_{\phi \otimes \bar{\psi}}(i(h,1))= \langle \omega(h)(\phi), \psi \rangle, \qquad h \in H, \ \phi, \psi \in \s.$$
So plugging this in, we have
\begin{align*}  
\mathcal{E}(x \otimes \overline{z}, \ y \otimes \overline{u}) &= \int_H \langle \omega(h)(x),z \rangle \overline{\langle \omega(h)(y),u \rangle} \mathrm{d}h \\ 
&=\int_H \langle x,\omega(h^{-1})(z) \rangle \overline{\langle y,\omega(h^{-1})(u) \rangle} \mathrm{d}h \\
&=\int_H \langle x,\omega(h)(z) \rangle \overline{\langle y,\omega(h)(u) \rangle} \mathrm{d}h \\
&=P_H(x \otimes \overline{u} \otimes \ y \otimes \overline{z}).
\end{align*}
Gan and Ichino prove (\cite[Thm. 17.2]{Gan-Ichino-14}) that the space ${\rm Hom}_{G \times H\times H}(\bm{\omega} \otimes \overline{\bm{\omega}}\otimes \bar{\chi}_V \otimes \chi_V, {\bf 1})$ is one dimensional, so that $\mathcal{I}$ and $\mathcal{E}$ are proportional. The proportionality constant $C$ depends on the choice of Haar measures on $G$ and $H$. For a specific choice of Haar measures (\cite[Section 20.1]{Gan-Ichino-14}, Gan and Ichino calculate that $C=1$ in the unitary/unitary case, and in the metaplectic/orthogonal case, $C=2$ or $C=1/2$ depending on $\varepsilon$  (see \cite[Thm. 20.1]{Gan-Ichino-14}). Scaling the Haar measure they use by a factor of $1/2$ or $2$, we make $C=1$ in the metaplectic/orthogonal case. This gives us 
$$P_G(x \otimes \overline{u} \otimes \ y \otimes \overline{z}) =\mathcal{I}(x \otimes \overline{z}, \ y \otimes \overline{u}) = 
\mathcal{E}(x \otimes \overline{z}, \ y \otimes \overline{u}) = P_G(x \otimes \overline{u} \otimes \ y \otimes \overline{z})
$$
as desired.
\end{proof}

For the remainder of the paper, we fix our Haar measures as in the proof of Prop. \ref{poisson-proof}. Putting together Propositions \ref{construct-1}, \ref{construct-2}, \ref{poisson-proof} and \ref{constructMorita} gives us the following.
\begin{theorem}
\label{localbimodule}
The space $\s$ is a nondegenerate inner product bimodule for $(\mathcal{S}(G),\mathcal{S}(H))$. Its $C^{*}$-module completion is a Hilbert $C^{*}$-bimodule for $(C^{*}_{r}(G),C^{*}_{r}(H))$.
\end{theorem}

We will denote this Hilbert $C^*$-module by $\T$ and will call it the {\bf oscillator bimodule} as an homage to Rieffel's Heisenberg module.

\section{The $\theta$-subalgebras and induced representations}

\subsection{Truncation} 
We will now apply the final statement of Proposition \ref{constructMorita} to make $\T$ into an equivalence bimodule for the ideals $C^{*}_{\theta}(G)\subset C^{*}_{r}(G)$ and $C^{*}_{\theta}(H)\subset C^{*}_{r}(H)$ generated by the span of the left- and right inner products, respectively. We will then analyse their spectra in terms of the local theta correspondence.
\subsubsection{Truncate $C^*_r(H)$} 

As mentioned in Section \ref{hilbert-modules}, the linear span of the range of the form $\langle {\cdot}, {\cdot} \rangle\subrangle{H}$ is a two-sided ideal of $\mathcal{S}(H)$. Let us denote this linear span by $\langle \s, \s \rangle\subrangle{H}$. Consider its $C^{*}$-closure
\begin{equation} \label{shrink-RHS}  C^{*}_{\theta}(H):= \overline{\langle \s, \s \rangle\subrangle{H}}^{C^{*}_{r}(H)}.
\end{equation}
Thus $C^{*}_{\theta}(H)$ is closed two-sided ideal of $C^{*}_{r}(H)$.

\begin{proposition} \label{spectrum-B} The spectrum of $C^{*}_{\theta}(H)$ can be identified with the set of tempered irreducible representations $\pi$ of $H$ for which 
$\theta(\pi) \not= 0$.
\end{proposition}

\begin{proof} Basic theory tells us that the spectrum of $C^{*}_{\theta}(H)$ is simply the subset of the spectrum of $C^*_r(H)$ made precisely of those elements which do not vanish on $C^{*}_{\theta}(H)$. 

Let $\pi$ be a tempered irreducible representation of $H$ (in other words, an element of the spectrum of $C^*_r(H)$). Observe that $\pi$ vanishes on $C^{*}_{\theta}(H)$ if and only if it vanishes on the range of $\langle {\cdot}, {\cdot} \rangle\subrangle{H}$, thanks to density of the latter in the former. Let $x, x' \in \s$. Then $\pi(\langle x,x' \rangle\subrangle{H})$ is the zero operator on $V_\pi$ if and only if 
$$\bigl \langle v, \pi(\langle x,x' \rangle\subrangle{H})(v') \bigr \rangle = 0$$
for all $v,v' \in V_\pi$. As seen in the proof of Prop. \ref{construct-1}, we have
$$ \left \langle v, \pi(\langle x, x' \rangle\subrangle{H})(v') \right \rangle =(x{\otimes}v, x' {\otimes}v')_\pi$$
where $({\cdot},{\cdot})_\pi$ is the Hermitian form on $\s \otimes V_\pi$ (see \ref{Li-form}). Therefore, we conclude that $\pi$ vanishes on $C^{*}_{\theta}(H)$ if and only the form $({\cdot},{\cdot})_\pi$, hence $L(\pi)$ is zero. However, by Prop. \ref{form-positive}, $L(\pi)$ is isomorphic to $\theta(\pi^*)=\Theta(\pi^*)$. The claim now follows from the fact that $\Theta(\pi^*)$ is non-zero if and only if $\Theta(\pi)$ is non-zero, an immediate corollary of Lemma 6.1 of \cite{Gan-Ichino-14}.
\end{proof}

\subsubsection{Truncate $C^*_r(G)$} The linear span of the range of $\prescript{}{G}{\langle \s, \s \rangle}$ is a two-sided ideal of $\mathcal{S}(G)$. Consider its  $C^{*}$-closure
\begin{equation} \label{shrink-LHS}  C^{*}_{\theta}(G):= \overline{\prescript{}{G}{\langle \s, \s \rangle}}^{C^{*}_{r}(G)}.
\end{equation}
Then $C^{*}_{\theta}(G)$ is a closed two-sided ideal of $C^*_r(G)$, and hence is a $C^*$-subalgebra.

\begin{proposition} \label{spectrum-A} The spectrum of $C^{*}_{\theta}(G)$ can be identified with tempered irreducible (necessarily genuine) representations $\pi$ of $G$ for which 
$\theta(\pi) \not= 0$.
\end{proposition}

\begin{proof} The proof is the same as that of Prop. \ref{spectrum-B}. Given tempered irreducible representation $\pi$ of $G$, $x,x' \in \s$ and $v,v' \in \pi$, we observe that
\begin{align} \left \langle v, \pi(\prescript{}{G}{\langle x,x' \rangle})(v') \right \rangle 
=  \int_{G} \langle \omega_{g}(x), x' \rangle \langle v, \pi_g(v') \rangle \mathrm{d}g 
=  \int_{G} \overline{\langle x, \omega_{g}(x')\rangle} \langle v, \pi_g(v') \rangle \mathrm{d}g.
\end{align}
This is the conjugate of the Hermitian form $({\cdot},{\cdot})_{\overline{\pi}}$ on $\s \otimes V_{\overline{\pi}}$ (see \ref{Li-form}) where $\overline{\pi}$ is the conjugate representation on $\overline{V_\pi}$. Therefore, we conclude that $\pi$ vanishes on $C^{*}_{\theta}(G)$ if and only if $({\cdot},{\cdot})_{\overline{\pi}}$, hence $({\cdot},{\cdot})_{\pi}$, is zero.  The claim now follows from Prop. \ref{form-positive} as explained in the proof of Prop. \ref{spectrum-B}.
\end{proof}

\subsection{The induced $G$-representation} \label{induced-G-rep}

Consider the action of $G$ on $\s$ via the oscillator representation. As the action of $G$ and $H$ commute, the $G$-action preserves the $C^{*}_{\theta}(H)$-valued inner product $\langle {\cdot},{\cdot}\rangle\subrangle{H}$ on $\s$, that is
\[\langle\omega(g)x,\omega(g)y\rangle_{H}=\langle x,y\rangle_{H},\quad x,y\in\s.\]
 It follows that $\|\omega(g)\|_{\textnormal{End}^{*}(\Theta)}=1$, so $\omega(g)$ can be extended to a unitary operator on all of $\T$.

Now, given an irreducible representation of $C^{*}_{\theta}(H)$, in other words, a tempered irreducible representation $(\pi,V_\pi)$ of $H$ with $\theta(\pi)\not= 0$, consider $\T \otimes_{C^{*}_{\theta}(H)} V_\pi=\T \otimes_{C^{*}_r(H)} V_\pi$. Following the previous paragraph, $G$ acts on $\T \otimes_{C^{*}_{\theta}(H)} V_\pi$ via the formula 
$$g\cdot (x \otimes v) := \omega(g)(x) \otimes v$$
where $g \in G$, $x \in \T$ and $v \in V_\pi$. Recall from Section \ref{gen-ind} that the space $\T \otimes_{C^{*}_{\theta}(H)} V_\pi$ comes equipped with a positive Hermitian form
$$ \left (x \otimes v, x' \otimes v' \right ) := \left \langle v, \pi(\langle x,x' \rangle\subrangle{B})(v') \right \rangle\subrangle{V_\pi}.
$$
As the action of $G$ commutes with that of $C^{*}_{\theta}(H)$ on $\T$, the above form and also its radical are preserved under the action of $G$ and we obtain a unitary representation of $G$ on the Hilbert space $\T \otimes_{C^{*}_{\theta}(H)} V_\pi$ which we will denote
$$\textnormal{\small Ind}_H^G(\T,\pi).$$

\begin{proposition} \label{induced-G-rep-theta} The unitary representation $\textnormal{\small Ind}_H^G(\T,\pi)$ of $G$ is precisely the unitarization of $\theta(\pi^*)$. 
\end{proposition}
\begin{proof} By Prop. \ref{form-positive}, we can replace $\theta(\pi^*)$ by $L(\pi)$. Consider the map 
$$Z: \s \otimes V_{\pi}^\infty \longrightarrow \T \otimes_{C^{*}_{\theta}(H)} V_{\pi}$$
given by
$$x \otimes v \mapsto x \otimes_{C^{*}_{\theta}(H)} v$$
where we view $\s$ as a dense subspace of $\T$. 

As we have already observed in the proof of Prop. \ref{spectrum-B}, we have
$$\left (x \otimes v, x' \otimes v' \right)_{\pi} = \int_H \langle \omega(h)(x),x' \rangle \langle \pi(h)(v),v' \rangle_{V_{\pi}}\mathrm{d}h = \left \langle v, \pi(\langle x,x' \rangle\subrangle{B})(v') \right \rangle\subrangle{V_\pi}.$$
Thus the map $Z$ preserves the forms on the two sides. Therefore, the kernel of $Z$ is precisely the radical $N$ of $({\cdot},{\cdot})_{\pi}$, so that $Z$ descends to a linear embedding 
$$\left (\s \otimes V_{\pi}^\infty \right )/N \hookrightarrow \T \otimes_{C^{*}_{\theta}(H)} V_{\pi}.$$
Recall from Section \ref{Li-form} that the left hand side is precisely $L(\pi)$ and that the $G$-action on $L(\pi)$ is defined solely via the action of $G$ on $\s$ via the oscillator representation. Therefore the map $Z$ gives us the desired injective $G$-intertwiner.
\end{proof}

We summarize our results.
\begin{theorem} The oscillator bimodule $\T$ is an equivalence $(C^{*}_{\theta}(G),C^{*}_{\theta}(H))$-bimodule. Moreover, the associated induction map 
 $$\pi \mapsto \textnormal{\small Ind}_{C^{*}_{\theta}(H)}^{C^{*}_{\theta}(G)}(\T, \pi)$$
 captures the tempered local theta correspondence in the sense that if $\pi$ is a tempered irreducible representation of $H$ then $\textnormal{\small Ind}_{C^{*}_{\theta}(H)}^{C^{*}_{\theta}(G)}(\T, \pi)$ is (the integrated form of) the unitarization of $\theta(\pi^*)$. 
\end{theorem}
\begin{proof}
The Morita equivalence statement follows from Prop. \ref{constructMorita} and Thm. \ref{localbimodule}. For the induction part, let $\pi$ be a tempered irreducible representation of $H$. Recall that $\pi$ belongs to the spectrum of $C^{*}_{\theta}(H)$ (i.e. $\pi$ restricted to $C^{*}_{\theta}(H)$ is not zero) if and only if $\theta(\pi)\not=0$. If $\theta(\pi)= 0$ then $\pi(C^{*}_{\theta}(H))=0$ and hence $\textnormal{\small Ind}_{C^{*}_{\theta}(H)}^{C^{*}_{\theta}(G)}(\T, \pi)=0$. So we can assume that $\theta(\pi)\not=0$. We have seen in Prop. \ref{induced-G-rep-theta} that the $G$-representation $\textnormal{\small Ind}_H^G(\T,\pi)$ is the unitarization of $\theta(\pi^*)$. It is clear that the $C^{*}_{\theta}(G)$-representation $\textnormal{\small Ind}_{C^{*}_{\theta}(H)}^{C^{*}_{\theta}(G)}(\T, \pi)$ is nothing but the integrated form of $\textnormal{\small Ind}_{H}^{G}(\T,\pi)$. Therefore, the induction of representations of $C^{*}_{\theta}(H)$ to $C^{*}_{\theta}(G)$ implemented via $\T$ captures the local theta correspondence as claimed.
\end{proof}

\subsection{Functoriality} The induction of representations implemented by $\T$ establishes an equivalence between the categories of representations of the ideal $C^{*}_{\theta}(G)$ of $C^*_r(G)$ and of the ideal $C^{*}_{\theta}(H)$ of $C^*_r(H)$. Recall that the spectra of $C^{*}_{\theta}(G)$ and $C^{*}_{\theta}(H)$ capture those tempered irreducible representations of $G$ and $H$ which enter the theta correspondence and that the induction map, once restricted to the irreducible representations, captures the theta correspondence. Therefore it follows from Section \ref{sec: functorilaity} that the theta correspondence is functorial. 

\subsection{Continuity} \label{feature:continuity} The $C^*$-algebras $C^{*}_{\theta}(G)$ and $C^{*}_{\theta}(H)$ are strongly Morita equivalent and thus, by Section \ref{sec: homeomorphic}, their spectra are homeomorphic. In other words,  tempered theta correspondence is a homeomorphism with respect to the Fell topologies on each side.

\subsection{Support of the oscillator representation} In this section, we make some elementary observations regarding the role played by the oscillator representation in our picture of the theta correspondence. Recall that the oscillator representation in the equal rank case is tempered both as a $G$-representation and as an $H$-representation. We first show that the induction functor associated to the oscillator bimodule, when viewed as a $(C^*_r(G),C^*_r(H))$-correspondence, sends the regular representation of $H$ to the oscillator representation (viewed as a representation of $G$). Using this, we show that $C^*_\theta(G)$ sits as an ``essential ideal'' (definition below) in $C^*_\omega(G)$. Next we prove that the closure of the set of tempered representations that enter the theta correspondence equals the support of the oscillator representation. 

\begin{proposition}
\label{prop: reg-intertwine}
Denote by $(\omega,V_{\omega})$ the oscillator representation and by $(\rho, L^{2}(H))$ the left regular representation of $H$. 
The map
\begin{align*}
\mathcal{U}_{\T}: \s&\otimes^{\mathrm{alg}} \mathcal{S}(H)\to \s\\
x&\otimes f\mapsto x{\cdot} f,
\end{align*}
induces a $G$-equivariant unitary isomorphism
\begin{align*}
U_{\T}:\T\otimes_{C^{*}_{r}(H)}L^{2}(H)&\to V_{\omega}.
\end{align*}
\end{proposition}
\begin{proof}
The map $\mathcal{U}_{\T}$ has dense range since the unit element of $H$ acts as the identity operator on $V_{\omega}$ and thus for any approximate unit $u_{n}\in C^{*}_{r}(H)$ and $x\in \s$ the sequence $x\cdot u_{n}$ converges to $x$ in norm in $V_{\omega}$. It thus suffices to show that $\mathcal{U}_{\T}$ is an isometry. This is established by the following calculation.
\begin{align*}
\langle \mathcal{U}_{\T}(x\otimes f), \mathcal{U}_{\T}(y\otimes g)\rangle &= \langle x\cdot f, y\cdot g\rangle=\int_{H}\int_{H}\langle \omega(s)f(s^{-1})x, \omega(t)g(t^{-1})y\rangle \mathrm{d}s\mathrm{d} t\\
&=\int_{H}\int_{H}\langle f(s^{-1})x, \omega(t)g(t^{-1}s^{-1})y\rangle \mathrm{d}s\mathrm{d} t \\
&=\int_{H}\int_{H}\langle f(s)x, \omega(t)g(t^{-1}s)y\rangle \mathrm{d}s\mathrm{d} t\\
&=\int_{H}\int_{H}\overline{f(s)}\langle x, \omega(t)y\rangle g(t^{-1}s) \mathrm{d}s\mathrm{d} t \\ 
&=\int_{H}\int_{H}\overline{f(s)}\langle x, \omega(t)y\rangle (\rho(t)g)(s) \mathrm{d}s\mathrm{d} t\\
&=\int_{H}\overline{f(s)}\rho(\langle x, y\rangle_{\T})g(s) \mathrm{d}s\mathrm{d} t=\langle f, \rho(\langle x, y\rangle_{\T})g\rangle_{L^{2}(H)}.
\end{align*}
The $G$-equivariance now follows since $\mathcal{U}_{\T}(gx\otimes f)=(gx)\cdot f=g(x\cdot f)=g(\mathcal{U}_{\T}(x\otimes f))$.
\end{proof}

In accordance with the notation we introduced in Section \ref{C*-algebras}, let $C^*_{\omega}(G)$ denote the image of $C^*(G)$ under the oscillator representation $\omega$. Recall that for an ideal $I\subset A$ in a $C^{*}$-algebra  its \emph{annihilator} is the set
\[I^{\perp}:=\left\{a\in A: \forall x\in I\,\,ax=0\right\},\]
and the ideal $I$ is \emph{essential} if $I^{\perp}=0$.
\begin{corollary}\label{inclusion} Denote by $L_{\T}:C^{*}_{r}(G)\to \textnormal{End}^{*}(\T)$ the $*$-homomorphism determined by the left module action. We have $\ker L_{\T}=\ker \omega$ and $C^{*}_{\theta}(G)$ embeds into $C^{*}_{\omega}(G)$ as a closed two sided essential ideal. In particular the support of $\omega$ contains $\widehat{C^{*}_{\theta}(G)}.$
\end{corollary}
\begin{proof} The regular representation of $H$ is faithful, so the unitary $U_{\T}$ from Proposition \ref{prop: reg-intertwine} induces an injection
\[ J_{\T}:\textnormal{End}^{*}(\T)\to B(\T\otimes_{C^{*}_{r}(H)}L^{2}(H))\simeq B(V_{\omega}).\]
As $U_{\T}$ is $G$-equivairant, $J_{\T}$ satisfies \[(J_{\T}\circ L_{\T})(C^{*}_{r}(G))=\omega(C^{*}_{r}(G))=C^{*}_{\omega}(G),\] and since $J_{\T}$ is injective we find that $\ker L_{\T}=\ker\omega$. Furthermore, $\T$ carries a left $C^{*}_{\theta}(G)$-valued inner product, and thus $J_{\T}\circ L_{\T}$ restricts to an injection $C^{*}_{\theta}(G)\to C^{*}_{\omega}(G)$. In particular $\ker \omega\cap C^{*}_{\theta}(G)=0$ and $\ker \omega\subset C^{*}_{\theta}(G)^{\perp}$, where the latter denotes the annihilator of $C^{*}_{\theta}(G)$. 

For $b\in C^{*}_{r}(G)$ and $x\in \s$ we have $bx\in \Theta\cap V_{\omega}$ since $b$ can be approximated in norm by a sequence $b_{n}\in \mathcal{S}(H)$, so that $ b_{n}x\to bx$ in both $V_{\omega}$ and $\Theta$. For $b\in C^{*}_{\theta}(G)^{\perp}$ and $x\in \s$ we have 
$$0=b^{*}   \prescript{}{G}{\langle x,x\rangle} b=\prescript{}{G}{\langle bx, bx \rangle}.$$ 
Since the left $C^{*}_{r}(G)$-valued inner product is nondegenerate we find that $bx=0\in V_\omega\cap\T$ and since $\s$ is dense in $V_{\omega}$ we have $b\in\ker\omega$. The statement follows.
\end{proof}

Recall that the closure of $\widehat{C^{*}_{\theta}(G)}$ in the tempered dual of $G$ is the set of those irreducible representations of $G$ which are weakly contained in $\widehat{C^{*}_{\theta}(G)}$. 

\begin{lemma}\label{containment} The oscillator representation $\omega$ is contained in the closure of $\widehat{C^{*}_{\theta}(G)}$.
\end{lemma}
\begin{proof} We need to show that (see (\ref{weak-containment}))
$$\bigcap_{\pi \in \widehat{C^{*}_{\theta}(G)}}\ker\pi\subset \ker \omega.$$ 
Let $b\in \bigcap_{\pi \in \widehat{C^{*}_{\theta}(G)}}\ker\pi$, $x\in\s$ and consider $b\cdot x\in\T\cap V_{\omega}$. Then for all $\pi\in\widehat{C^{*}_{\theta}(G)}$ we have
\[\pi(_{G}\langle b\cdot x, b\cdot x\rangle)=\pi(b^{*})\pi(_{G}\langle x,x\rangle)\pi(b)=0,\]
so it follows that $_{G}\langle b\cdot x, b\cdot x\rangle=0\in C^{*}_{\theta}(G)$. Therefore $b\cdot x=0$ for all $x\in \s$, that is $b\in\ker \omega$. 
\end{proof}
\begin{corollary} The closure of $\widehat{C^{*}_{\theta}(G)}$ in the tempered dual of $G$ is equal to the support of the oscillator representation $\omega$.
\end{corollary}
\begin{proof} Note that the support of $\omega$ is simply the closure of the singleton $\{ \omega \}$. Thus Lemma \ref{containment} gives us that the support of $\omega$ is contained in the closure of $\widehat{C^{*}_{\theta}(G)}$. The converse containment is given by Corollary \ref{inclusion}. 
\end{proof}

The above corollary is known; it follows alternatively from Thm. 3.0.2 of Sakellaridis' paper \cite{Sakellaridis-17} (see his Remark 3.0.3).


\section{Application: transfer of characters} 
\subsection{} Let $\pi$ be an irreducible representation of $C^{*}_{\theta}(H)$. For convenience, let us temporarily introduce the notations, 
$$\textnormal{\small Ind}(\pi) :=\textnormal{\small Ind}_{C^{*}_{\theta}(H)}^{C^{*}_{\theta}(G)}(\T,\pi), \quad 
V_{\textnormal{\small Ind}(\pi)}:=  \T \otimes_{C^{*}_{\theta}(H)} V_\pi.$$ 
Consider the map
$$T : \T \to \mathcal{L}(V_\pi, V_{\textnormal{\small Ind}(\pi)}), \qquad T(x)(v):= x\otimes v$$
for $x \in \T$ and $v \in V_\pi$. The map $T$ is linear and satisfies\footnote{We are essentially considering the $C^*$-counterpart of ${\rm Hom}_{G\times H}(\omega, \pi {\otimes} \theta(\pi^*))$. The intertwiner space 
 ${\rm Hom}_{G\times H}(\omega, \pi {\otimes} \theta(\pi))$ is one dimensional as a consequence of Howe duality (see \cite[p.138 Remark (iii)]{Gelbart-93}).} (see e.g. \cite[Lemma 2.6]{Zettl-82})
$$T(a{\cdot} x{\cdot}b) =  \textnormal{\small Ind}(\pi)(a)\ T(x) \ \pi(b)$$
for all $a \in C^{*}_{\theta}(G)$ and $b \in C^{*}_{\theta}(H)$. Moreover, for $x,y \in T$ and $w,w' \in V_\pi$, we have 
$$\langle T(x)(w'), y \otimes w \rangle\subrangle{V_{\textnormal{\small Ind}(\pi)}}=\langle x\otimes w', y \otimes w \rangle\subrangle{V_{\textnormal{\small Ind}(\pi)}}
=\langle w', \pi(\langle x,y \rangle\subrangle{H})(w) \rangle\subrangle{V_{\pi}}$$
so that
$$T(x)^*(y\otimes w) = \pi(\langle x,y \rangle\subrangle{H})(w).$$ 
Moreover,
\begin{equation} \label{adjoint} T(x)^*T(y)=\pi(\langle x,y \rangle\subrangle{H}), \qquad T(y)T(x)^*=\textnormal{\small Ind}(\pi)(\prescript{}{G}{\langle y,x \rangle}).
\end{equation}
To see the latter, observe that 
\begin{eqnarray*}T(y)T(x)^*(z\otimes w) &=& y\otimes \pi(\langle x,z \rangle\subrangle{H}))(w) \\
&=& y{\cdot}\langle x,z \rangle\subrangle{H} \otimes w \\
&=& \prescript{}{G}{\langle y,x \rangle}{\cdot} z \otimes w \\
&=& \textnormal{\small Ind}(\pi)(\prescript{}{G}{\langle y,x \rangle})(z \otimes w).
\end{eqnarray*}

\subsection{} Recall that the elements $\langle x,y \rangle\subrangle{H}$ and $\prescript{}{G}{\langle y,x \rangle}$ lie in the Schwartz algebras of $H$ and $G$ respectively. As both $\pi$ and $\textnormal{\small Ind}(\pi)$ are tempered, the operators $\pi(\langle x,y \rangle\subrangle{H})$ and $\textnormal{\small Ind}(\pi)(\prescript{}{G}{\langle y,x \rangle})$ are of trace class.

\begin{lemma} Let $x,y \in \s$. We have
$${\rm tr} \ \pi(\langle x,y \rangle\subrangle{H}) = {\rm tr} \ \textnormal{\small Ind}(\pi)(\prescript{}{G}{\langle x,y \rangle}).$$
\end{lemma} 
\begin{proof} The case where $x=y$ can be found in  \cite[Cor. 5]{ARW-07}: it follows directly from (\ref{adjoint}) together with the fact that the traces of the operators $SS^*$ and $S^*S$ are the same for any $S \in \mathcal{L}(V_\pi, V_{\textnormal{\small Ind}(\pi)})$. For the case $x\not= y$, one uses the polarization identity
$$4 \langle x,y \rangle\subrangle{H} = \sum_{k=0}^3 i^k \langle x+i^k y, x+i^k y \rangle\subrangle{H}$$
to reduce to the case where $x=y$. 
\end{proof}

Now recall that $\theta(\pi)=\textnormal{\small Ind}(\pi^*)$ so that for every $x,y \in \s$, we have
$${\rm tr} \ \pi^*(\langle x,y \rangle\subrangle{H})= {\rm tr} \ \theta(\pi)(\prescript{}{G}{\langle x,y \rangle}).$$
Observing that $\pi^*(\langle x,y \rangle\subrangle{H}) = \pi(\overline{\langle x,y \rangle\subrangle{H}})=\pi(\langle y,x \rangle\subrangle{H})$, we obtain the following corollary where we use the terminology of Section \ref{char}.

\begin{corollary}\label{cor: transferofcharacters} Let $\pi$ be a tempered irreducible representation of $H$ that enters the theta correspondence. Given $x,y \in \s$, let $\langle x,y \rangle\subrangle{H} \in \mathcal{S}(H)$ and $\prescript{}{G}{\langle x,y \rangle} \in \mathcal{S}(G)$ be the matrix coefficient functions defined earlier in (\ref{right-inner-product}) and (\ref{left-inner-product}). We have
 $${\rm ch}(\theta(\pi))(\prescript{}{G}{\langle x, y \rangle}) = {\rm ch}(\pi)(\langle y,x \rangle\subrangle{H}).$$
\end{corollary}

The above result has been recently announced by Wee Teck Gan \cite{Gan-20, Gan-21}. While his proof seems different than ours, it can be said that it philosophically agrees with ours in that matrix coefficients of the oscillator representation play a central role.

\section{Application: preservation of formal degrees} \label{App-2}
In the equal rank local theta correspondence, discrete series representations are sent to discrete series representations, see \cite{Gan-Savin-12}. In this section, we will reprove a well-known result of Gan and Ichino about the preservation of formal degrees of discrete series. The main point of interest will be our method which will feature $K$-theory and transfer of trace maps. 

Let $(G,H)$ be an equal rank dual pair as in Section \ref{LTC}. We will assume in this section that in the metaplectic-orthogonal case, $H$ denotes the orthogonal group $O(V)$. 

Recall that an irreducible unitary representation of $H$ is called {\em discrete series} if its matrix coefficients lie in $L^2(H)$\footnote{Recall that $H$ has compact centre, hence square integrability modulo the centre equals square integrability in the above sense.}.  The {\em formal degree} of a discrete series representation $\pi$ of, say, $H$ is the positive real number ${\rm deg}(\pi)$ such that 
$$\int_H \langle v, \pi(h)(v') \rangle \overline{\langle w, \pi(h)(w') \rangle}\mathrm{d}h= \dfrac{1}{{\rm deg}(\pi)} \langle v, w \rangle \langle v', w' \rangle$$
for all $v,v',w,w' \in V_{\pi}$ (see e.g. \cite[14.3.3]{Dixmier}). Note that the formal degree depends on the chosen Haar measure $dh$. Gan and Ichino proved in \cite{Gan-Ichino-14} there exists a choice of Haar measures on $G$ and on $H$ such that for every discrete series representation $\pi$ of $H$ which enters the theta correspondence, we have
\begin{equation} \label{degree-eq} {\rm deg}(\pi) = {\rm deg}(\theta(\pi)).
\end{equation}

\subsection{$K$-theory} Given a unital complex algebra $C$, one defines the abelian group $K_0(C)$ as the group of (formal differences of) Murray-von Neumann equivalence classes of idempotents in 
$$M_\infty(C) := \varinjlim_n M_n(C).$$ 
When $C$ is a $C^*$-algebra, we can alternatively describe $K_0(C)$ using (homotopy classes of) projections instead of idempotents. For a non-unital complex algebra $C$, one defines $K_0(C)$ as the kernel of map 
$$K_0(C^+) \to K_0(\C)\simeq \mathbb{Z}$$ 
induced by the natural map $C^+ \to \C$ where $C^+$ is the unitisation of $C$. 
Note that $K_0$ of a $C^*$-algebra is a naturally ordered group.

\subsection{Discrete series and $K$-theory} 

Let $\pi$ be a discrete series representation of $H$ such that $\theta(\pi)\not=0$, so that $\pi$ belongs to the spectrum of $C^{*}_{\theta}(H)$. In order to show that $\pi$ defines an element in $K_{*}(C^{*}_{\theta}(H))$, we need the following lemma.
\begin{lemma}\label{lem: clopen} The singleton $\{ \pi \}$ is a clopen (closed and open) subset of the spectrum of $C^{*}_{\theta}(H)$. 
\end{lemma} 
\begin{proof} As $H$ is a reductive $p$-adic (separable) group, it is liminal\footnote{This means that $\pi(a)$ is a compact operator for every $a \in C^*(H)$ and for every irreducible unitary representation $\pi$ of $H$.}. This implies (see \cite[9.5.3]{Dixmier}) that $\{ \pi \}$ is closed in the unitary dual of $H$, which in turn implies closedness in the spectrum of $C^{*}_{\theta}(H)$. It remains to prove that $\{ \pi \}$ is also open. 

It is a well-known fact (see \cite[18.4.2]{Dixmier}) that if $\pi$ is integrable\footnote{This is the case for ``most'' discrete series.}, then $\{\pi \}$ is open in the tempered dual of $H$ (and hence in the spectrum of $C^{*}_{\theta}(H)$). To argue that $\{\pi \}$ is open for a non-integrable discrete series representation $\pi$, we will appeal to Harish-Chandra's work.

In case $(G,H)$ is a unitary dual pair, $H$ is a connected reductive $p$-adic group with compact center. The Plancherel formula of Harish-Chandra (\cite{Harish-Chandra, Waldspurger-03}) gives a description of the connected components of the tempered dual $\widehat{H}_{temp}$ of $H$ and it follows this description that each $\{ \pi \}$ forms a connected component of $\widehat{H}_{{\rm temp}}$ (i.e. it is clopen) for every discrete series representation\footnote{Note that compactness of the center is important here; indeed, when the center is noncompact, one can place discrete series in continuous families by twisting them with suitable characters of the center. This is already visible in the case of ${\rm GL(2,\R)}$.}. Since the topology of $\widehat{C^{*}_{\theta}(H)}$ is simply the subspace topology inherited from $\widehat{H}_{{\rm temp}}$, the point $\{ \pi \}$ is clopen in $\widehat{C^{*}_{\theta}(H)}$ as well.  

Let us now consider the case $(G,H)=(Mp_{2n},O_{2n+1})$, so that $H=O(V)$ with $V$ odd dimensional. We have  $O(V) \simeq SO(V) \times \{ \pm 1 \}$. The restriction map gives us a $2$-to-$1$ surjection from the tempered dual of $O(V)$ to that of $SO(V)$. Restriction of $\pi$ to $SO(V)$, say $\sigma$, is again a discrete series representation and the Plancherel formula argument above tells us that $\{ \sigma \}$ is clopen in the tempered dual of $SO(V)$. As the restriction map is continuous (\cite[Lemma 1.11]{Fell-64}), the preimage of $\{ \sigma \}$ is clopen in the tempered dual of $O(V)$. It is well-known (\cite[Prop. 6.3]{Gan-Savin-12}) that in this preimage, which has size two, only $\pi$ enters the theta correspondence. Therefore this preimage, which is clopen, intersects the spectrum of $C^{*}_{\theta}(H)$ only in the singleton $\{ \pi \}$ giving us the claim.
\end{proof}

The fact that $\{ \pi \}$ is a clopen subset of the spectrum of $C^{*}_{\theta}(H)$ implies that the closed two-sided ideal ${\rm ker}(\pi)$ (here $\pi$ is restricted to $C^{*}_{\theta}(H)$) is complemented (see \cite{Valette-84}):
$$C^{*}_{\theta}(H) \simeq {\rm ker}(\pi) \oplus J_\pi,$$
where $J_\pi$ is the closed two-sided ideal  
$$J_\pi:= \bigcap_{\substack{\sigma \in \widehat{C^{*}_{\theta}(H)} \\ \sigma \not= \pi}} {\rm ker}(\sigma),$$
and the sum is a direct sum of $C^*$-algebras. 

As the group $H$ is liminal (see proof of Lemma \ref{lem: clopen}), $\pi : C^*_r(H) \to \mathbb{K}(V_\pi)$ is surjective. Since its restriction to $C^{*}_{\theta}(H)$ is non-zero and $ \mathbb{K}(V_\pi)$ is simple, we conclude that $\pi : C^{*}_{\theta}(H) \to \mathbb{K}(V_\pi)$ is still surjective. It follows that 
$$J_\pi \simeq C^{*}_{\theta}(H)/ {\rm ker}(\pi) \simeq \mathbb{K}(V_\pi).$$
As $J_\pi$ is a direct summand, the injection $J_\pi \hookrightarrow C^{*}_{\theta}(H)$ leads to an injection 
$$\iota: K_0(J_\pi) \to K_0(C^{*}_{\theta}(H)).$$
Since $K_0(\mathbb{K}(V_\pi))$ is isomorphic to $\mathbb{Z}$ as an ordered abelian group\footnote{The isomorphism is canonical, sending the class of an idempotent to its trace.}, we conclude that there is a copy of $\mathbb{Z}$ in $K_0(C^{*}_{\theta}(H))$ that is contributed by $\pi$. We fix the positive generator $[\pi]$ of $K_0(J_\pi)\simeq \mathbb{Z}$ and call it {\bf the class associated to $\pi$} viewing it inside $K_0(C^{*}_{\theta}(H))$.

We have proven the following.
\begin{lemma}\label{lem: class} The discrete series representation $\pi$ defines a class $[\pi]\in K_0(C^{*}_{\theta}(H))$ of infinite order.
\end{lemma}

Now, we can use the oscillator bimodule to induce ideals as well (see Section \ref{ideals}), leading to an isomorphism of the lattices of ideal of $A$ and $B$. Consider the discrete series representation $\theta(\pi^*)$ of $G$. As we have shown (see Section \ref{feature:continuity}) that the theta correspondence induces a homeomorphism between the spectra of $C^{*}_{\theta}(G)$ and $C^{*}_{\theta}(H)$, we deduce that $\{ \theta(\pi^*) \}$ is isolated in $\widehat{C^{*}_{\theta}(G)}$. It follows from the previous paragraph that we have a direct sum of $C^*$-algebras 
$$C^{*}_{\theta}(G) \simeq {\rm ker}(\theta(\pi^*)) \oplus J_{\theta(\pi^*)}$$
with
$$J_{\theta(\pi^*)}:= \bigcap_{\substack{\sigma \in \widehat{C^{*}_{\theta}(G)} \\ \sigma \not= \theta(\pi^*)}} {\rm ker}(\sigma).$$
As the induction of ideals is compatible with the induction of representations, for any $\sigma \in \widehat{C^{*}_{\theta}(H)}$, we have 
$$\textnormal{\small Ind}_{C^{*}_{\theta}(H)}^{C^{*}_{\theta}(G)}(\T, {\rm ker}(\sigma)) = {\rm ker}(\textnormal{\small Ind}_{C^{*}_{\theta}(H)}^{C^{*}_{\theta}(G)}(\T, \sigma)).$$
It follows that
\begin{equation} \textnormal{\small Ind}(\T, {\rm ker}(\pi))={\rm ker}(\theta(\pi^*)), \qquad  \textnormal{\small Ind}(\T, J_\pi) = J_{\theta(\pi^*)}.
\end{equation}

A $(C,D)$-equivalence bimodule $X$ gives rise to an isomorphism of $K$-groups (see, for example, Prop. 2.4 and the paragraph following that in \cite{Rieffel-81} for the unital case): 
$$\Psi_X : K_0(C) \xrightarrow{\ \ \simeq \ \ } K_0(D)$$
as ordered groups. Assuming the set-up of the above paragraph, $\T_{\pi} := \T J_\pi$ is a  $(J_{\theta(\pi^*)}, J_{\pi})$-equivalence bimodule. We are led to following commutative diagram
\begin{equation} \label{comm-K-diag}
\xymatrix{K_0(J_{\theta(\pi^*)}) \ar[d]^{\Psi_{\T_{\pi}}}_{\simeq} \ar@{^{(}->}[r]^{\iota} & K_0(C^{*}_{\theta}(G)) \ar[d]^{\Psi_\T}_{\simeq} \\
K_0(J_\pi) \ar@{^{(}->}[r]^{\iota} & K_0(C^{*}_{\theta}(H)) 
} 
\end{equation}

\begin{lemma} \label{classes-match} Let $\pi$ be a discrete series representation of $H$ such that $\theta(\pi)\neq 0$. Then the class of $\theta(\pi^*)$ in $K_0(C^{*}_{\theta}(G))$ is taken to the class of $\pi$ in $K_0(C^{*}_{\theta}(H))$ under the map $\Psi_\T$.
\end{lemma}
\begin{proof} As $\Psi_{\T_\pi}$ is an isomorphism of ordered groups, it takes $[\theta(\pi^*)]$ to $[\pi]$. Our claim now follows from the commutativity of diagram (\ref{comm-K-diag}).
\end{proof}

\subsection{Traces} \label{traces} Let $C$ be a $C^*$-algebra and let $C_+$ denote its cone of positive elements. By a {\bf trace} on $C$, we mean a linear map $\chi: C_+ \to [0,\infty]$ such that $\chi(0)=0$ and $\chi(cc^*)=\chi(c^*c)$ for all $c \in C$. If $\chi$ takes finite values, then it is called {\em bounded}. In this case, $\chi$ can be extended to a linear functional on $C$ (since $C_+$ spans $C$) satisfying the usual trace property: $\chi(cd)=\chi(dc)$ for all $c,d \in C$. 

We will be mainly interested unbounded traces. We say that $\chi$ is densely defined if its domain $\{ c \in C_+ \mid \chi(c)< \infty \}$ is dense in $C_+$. 
The {\bf canonical trace} $\tau_H$ on $C^*_r(H)$ is a densely defined, unbounded trace that is determined uniquely by the property  
$\tau_H(f)=f(e)$ for any $f \in \mathcal{S}(H)$. 

Upon restriction, one obtains a linear map $\tau_H: \mathcal{S}(H) \to \C$ that satisfies the usual trace property. 
The map $\tau_H$ on the algebra $\mathcal{S}(H)$ induces\footnote{One can extend the trace to $M_n(\mathcal{S}(H))$ in the obvious way and this will give a well-defined map on the classes of projections.} a linear functional, denoted $\tau_H^*$, on $K_0(\mathcal{S}(H))$. It follows from  
Thm. \ref{HCS} (ii) that $\mathcal{S}(H)$ is spectral invariant in $C^*_r(H)$, which in turn implies that the inclusion $i:\mathcal{S}(H)\to C^{*}_{r}(H)$ induces an isomorphism $i_{*}:K_0(\mathcal{S}(H))\xrightarrow{\sim}K_0(C^*_r(H))$ in $K$-theory. Thus we obtain a linear map  
$$\tau_H^* : K_0(C^*_r(H)) \to \C.$$
They key fact is that (see e.g. \cite[Section 2.3]{Lafforgue}) if $\pi$ is a discrete series representation of $H$ then
\begin{equation} \label{deg-up-to-sign} \tau_H^*([\pi])= {\rm deg}(\pi).
\end{equation} 
It is worth mentioning that $\tau_H^*$ vanishes on all other classes which do not correspond to discrete series representations.

\subsection{Transfer of traces} Given a densely defined trace $\chi$ on $C^{*}_{\theta}(H)$, one can construct, using the oscillator bimodule, a densely defined trace $\widehat{\chi}$ on $C^{*}_{\theta}(G)$ which satisfies
$$\widehat{\chi}(\prescript{}{G}{\langle x,x \rangle}) = \chi(\langle x,x \rangle\subrangle{H})$$
for all $x \in \T$ (see \cite[Section 2]{Rieffel-81} for bounded traces over unital algebras, and \cite[Section 1, Prop. 1.3.11]{Pierrot-02} for the densely defined case, see also \cite[Section 2.1]{Combes-Zettl-83} for the case of `lower semi-continuous' traces which our canonical traces are examples of). For the canonical trace, we have
\begin{equation} \label{explicit-transfer} \widehat{\tau}_H(\prescript{}{G}{\langle x,x \rangle})= \tau_H(\langle x,x \rangle\subrangle{H})= \langle x,x \rangle = \tau_G(\prescript{}{G}{\langle x,x \rangle})
\end{equation}
for all $x \in \s$. Therefore $\widehat{\tau}_H$ equals the canonical trace $\tau_G$ on $C^{*}_{\theta}(G)$. 

We have the following cohomological aspect of the transfer of traces. Restrict $\tau_G, \tau_H$ to $C^{*}_{\theta}(G)$ and $C^{*}_{\theta}(H)$ respectively, and restrict $\tau_G^*,\tau_H^*$ to $K_0(C^{*}_{\theta}(G))$ and $K_0(C^{*}_{\theta}(H))$ respectively.  
\begin{lemma} \label{transfer-trace} The pull-back map $\tau_H^* \circ \Psi_\T : K_0(C^{*}_{\theta}(G)) \to \C$ agrees with the map $\widehat{\tau}_H^*: K_0(C^{*}_{\theta}(G)) \to \C$ associated to the transfer 
$\widehat{\tau}_H$ of $\tau_H$ from $C^{*}_{\theta}(H)$ to $C^{*}_{\theta}(G)$. 
\end{lemma}
\begin{proof} For bounded traces on unital algebras, this is already noted by Rieffel in Prop. 2.5 of \cite{Rieffel-81}. More generally, this is recorded by Pierrot in Cor. 1.3.12 of \cite{Pierrot-02}: to see this, just set his $A$ to be our $C^{*}_{\theta}(H)$ and his $E$ to be our $\T$ so that his ${\bf K}(E)$ becomes isomorphic our $C^{*}_{\theta}(G)$.
\end{proof}

\begin{corollary} \label{App-2-main} Let $\pi$ be a discrete series representation of $H$ such that $\theta(\pi)$ is non-zero. With the Haar measures for $G$ and $H$ chosen as discussed in Prop \ref{poisson-proof}, the formal degree of $\theta(\pi)$ equals that of $\pi$. 
\end{corollary}
\begin{proof} Consider the pull-back $\tau_H^* \circ \Psi_\T : K_0(C^{*}_{\theta}(G)) \to \C$. We have 
\begin{equation} \label{first-line} (\tau_H^* \circ \Psi_\T)([\theta(\pi^*)])= \tau_H^*([\pi])
\end{equation}
using Lemma \ref{classes-match}. By Lemma \ref{transfer-trace}, the map $\tau_H^* \circ \Psi_\T$ equals the map $K_0(C^{*}_{\theta}(G)) \to \C$ that is induced by the transfer $\widehat{\tau}_H$ of $\tau_H$ to $C^{*}_{\theta}(G)$. By (\ref{explicit-transfer}), we have $\widehat{\tau}_H=\tau_G$; thus
\begin{equation} \label{second-line} \tau_G^*([\theta(\pi^*)])=\widehat{\tau}_H^*([\theta(\pi^*)])=(\tau_H^* \circ \Psi_\T)([\theta(\pi^*)]).
\end{equation}
Recall from (\ref{deg-up-to-sign}) that $\tau_H^*([\pi])$ equals ${\rm deg}(\pi)$ and that $\tau_G^*([\theta(\pi^*)])$ equals ${\rm deg}(\theta(\pi^*))$. Thus, combining (\ref{first-line}) and (\ref{second-line}), we deduce that the degree of $\theta(\pi^*)$ equals that of $\pi$. The claim follows since $\pi^*$ and $\pi$ have the same degree.  
\end{proof}

We point out the Haar measures used above are the ones that we employed during the proof of Prop. \ref{poisson-proof}. 

\begin{remark} Here we only dealt with the canonical trace as we were interested in the formal degrees. Once could however consider the transfer of traces given orbital integrals (see e.g. \cite{Hochs-Wang-18}) associated to conjugacy classes other than the trivial element (which gives the canonical trace).
\end{remark}



\begin{thebibliography}{STF}

\bibitem{ARW-07} A. An Huef, I. Raeburn and D. Williams. \emph{Properties preserved under Morita equivalence of $C^*$-algebras.} Proc. Amer. Math. Soc. 135 (2007), no. 5, 1495--1503.



\bibitem{Beuzart-Plessis-20} R. Beuzart-Plessis. \emph{A local trace formula for the Gan-Gross-Prasad conjecture for unitary groups: the Archimedean case.} Ast\'erisque No. 418 (2020).

\bibitem{BlackadarOA} B. Blackadar, Operator algebras. Theory of $C^*$-algebras and von Neumann algebras. Encyclopaedia of Mathematical Sciences, 122. Operator Algebras and Non-commutative Geometry, III. Springer-Verlag, Berlin, 2006. xx+517 pp. 

\bibitem{Brodzki-Plymen-02} J. Brodzki and R. Plymen. \emph{Chern character for the Schwartz algebra of $p$-adic $GL(n)$}. Bull. London Math. Soc. 34 (2002), no. 2, 219--228.


\bibitem{Combes-Zettl-83} F. Combes and H. Zettl. \emph{Order structures, traces and weights on Morita equivalent $C^*$-algebras.} Math. Ann. 265 (1983), no. 1, 67--81.

\bibitem{Dixmier} J. Dixmier. \emph{$C^*$-algebras}. Translated from the French by Francis Jellett. North-Holland Mathematical Library, Vol. 15. North-Holland Publishing Co., Amsterdam-New York-Oxford, 1977.

\bibitem{Fell-64} J. M. G. Fell. {\em Weak containment and induced representations of groups. II.} Trans. Amer. Math. Soc. 110 (1964), 424--447.

Pure and Applied Mathematics, 126. Academic Press, Inc., Boston, MA, 1988.


\bibitem{Gan-Savin-12} W.T. Gan and G. Savin. \emph{Representations of metaplectic groups I: epsilon dichotomy and local Langlands correspondence.} Compos. Math. 148 (2012), no. 6, 1655--1694.

\bibitem{Gan-Ichino-14} W.T. Gan and A. Ichino. \emph{Formal degrees and local theta correspondence.} Invent. Math. 195 (2014), no. 3, 509--672.

\bibitem{Gan-Takeda-16} W.T. Gan and S. Takeda. \emph{A proof of the Howe duality conjecture.} J. Amer. Math. Soc. 29 (2016), no. 2, 473--493.

\bibitem{Gan-Takeda-16-2} W.T. Gan and S. Takeda. \emph{On the Howe duality conjecture in classical theta correspondence.} Advances in the theory of automorphic forms and their $L$-functions, 105--117, Contemp. Math., 664, Amer. Math. Soc., Providence, RI, 2016.

\bibitem{Gan-20} W.T. Gan. Video of the talk on 15 July 2020 at the talk on at the ``WIS Representation theory and Algebraic Geometry'' online seminar. \href{https://www.youtube.com/watch?v=qyl-oqBMm8E}{LINK} 

\bibitem{Gan-21} W.T. Gan. Video of a talk given in May 2021 at the conference ``Relative Aspects of the Langlands Program, L-Functions and Beyond Endoscopy" at CIRM.  \href{https://www.youtube.com/watch?v=fNo8vJAaM7w}{LINK}

\bibitem{Gan-22} W.T. Gan. Video of the first talk in a series of talks given in July 2022 at the conference ``Representations and Characters: Revisiting the Works of Harish-Chandra and Andr\'e Weil" at IMS. \href{https://mediaweb.ap.panopto.com/Panopto/Pages/Viewer.aspx?id=ad094691-1e62-42ca-9a17-aecc0015a126}{LINK}

\bibitem{Gelbart-93} S. Gelbart. \emph{On theta-series liftings for unitary groups.} in ``Theta functions: from the classical to the modern'', 129--174,
CRM Proc. Lecture Notes, 1, Amer. Math. Soc., Providence, RI, 1993.



\bibitem{Harish-Chandra} Harish-Chandra. \emph{Harmonic analysis on reductive $p$-adic groups.} in ``Harmonic analysis on homogeneous spaces'' (Proc. Sympos. Pure Math., Vol. XXVI, Williams Coll., Williamstown, Mass., 1972), pp. 167--192,
Proc. Sympos. Pure Math., Vol. XXVI, Amer. Math. Soc., Providence, RI, 1973.


\bibitem{Harris-Li-Sun} M. Harris, J.-S. Li, Jian-Shu and B. Sun. \emph{Theta correspondences for close unitary groups.} in ``Arithmetic geometry and automorphic forms'', 265--307, Adv. Lect. Math. (ALM), 19, Int. Press, Somerville, MA, 2011.



\bibitem{Hochs-Wang-18} P. Hochs and H. Wang.  \emph{A fixed point formula and Harish-Chandra's character formula.} Proc. Lond. Math. Soc. (3) 116 (2018), no. 1, 1--32.

\bibitem{Howe-79} R. Howe, \emph{$\theta$-series and invariant theory}, in ``Automorphic forms, representations and $L$-functions'' (Proc. Sympos. Pure Math., Oregon State Univ., Corvallis, Ore., 1977), Proc. Sympos. Pure Math., XXXIII, Amer. Math. Soc., Providence, R.I., 1979, pp. 275--285.

\bibitem{Howe-89} R. Howe. \emph{Transcending classical invariant theory.} J. Am. Math. Soc. 2, (1989) 535--552.

\bibitem{Ichino-22} A. Ichino. \emph{Theta lifting for tempered representations of real unitary groups.} Adv. Math. 398 (2022), Paper No. 108188, 70 pp.


\bibitem{Lafforgue} V. Lafforgue, \emph{Banach KK-theory and the Baum-Connes conjecture.} Proceedings of the International Congress of Mathematicians, Vol. II (Beijing, 2002), 795--812, Higher Ed. Press, Beijing, 2002

\bibitem{Lance} E.C. Lance, \emph{Hilbert $C^{*}$-modules- A toolkit for operator algebraists}, London Mathematical Society Lecture Note Series, 210. Cambridge University Press, Cambridge, 1995. x+130 pp.

\bibitem{Landsman-94} N.P. Landsman. \emph{Rieffel induction as generalized quantum Marsden-Weinstein reduction.} J. Geom. Phys. 15 (1995), no. 4, 285--319.


Progr. Math., 198, Birkh\"auser, Basel, 2001.

\bibitem{Li-89}  J.-S. Li. {\em Singular unitary representations of classical groups.} Invent. Math. 97 (1989), no. 2, 237--255.


\bibitem{Li-12} W.-W. Li. \emph{La formule des traces pour les rev$\hat{e}$tements de groupes r\'eductifs connexes. II.  Analyse harmonique locale.}  Ann. Sci. Ec. Norm. Sup\'er. (4) 45 (2012), no. 5, 787--859.

\bibitem{Loke-Przebinda-22} H. Y. Loke and T. Przebinda. \emph{The character correspondence in the stable range over a $p$-adic field.} \href{https://arxiv.org/abs/2207.07298}{ArXiv}.


\bibitem{Merino-20} A. Merino. \emph{Transfer of characters in the theta correspondence with one compact member.} J. Lie Theory 30 (2020), no. 4, 997--1026.

\bibitem{Phillips-91} N. C. Phillips. \emph{$K$-theory for Fr\'echet algebras.} Internat. J. Math. 2 (1991), no. 1, 77--129.

\bibitem{Pierrot-02} F. Pierrot. \emph{$K$-th\'eorie et multiplicit\'es dans $L^2(G/\Gamma)$.} M\'em. Soc. Math. Fr. (N.S.) No. 89 (2002).


\bibitem{Prasad-22} D. Prasad. Video of a talk given in April 2022 at the conference ``Minimal Representations and Theta Correspondence" at ESI. \href{https://www.youtube.com/watch?v=ZG_kj6D7XRs}{LINK}.

\bibitem{Przebinda-91} T. Przebinda. \emph{Characters, dual pairs, and unipotent representations.} J. Funct. Anal. 98 (1991), no. 1, 59--96.

\bibitem{Przebinda-00} T. Przebinda. \emph{A Cauchy Harish-Chandra integral, for a real reductive dual pair.} Invent. Math. 141 (2000), no. 2, 29--363.

\bibitem{Raeburn-Williams} I. Raeburn and D.P. Williams, \emph{Morita equivalence and continuous trace $C^{*}$-algebras}, AMS Mathematical Surveys and Monographs 60, 1998.


\bibitem{Rieffel-76} M.A. Rieffel. \emph{Commutation theorems and generalized commutation relations.} Bull. Soc. Math. France 104 (1976), no. 2, 205--224.

\bibitem{Rieffel-81} M.A. Rieffel. \emph{$C^*$-algebras associated with irrational rotations.} Pacific J. Math. 93 (1981), no. 2, 415--429.

\bibitem{Rieffel-88} M.A. Rieffel. \emph{Projective modules over higher-dimensional noncommutative tori.} Canad. J. Math. 40 (1988), no. 2, 257--338.

\bibitem{Sakellaridis-17} Y. Sakellaridis. \emph{Plancherel decomposition of Howe duality and Euler factorization of automorphic functionals.} in ``Representation theory, number theory, and invariant theory'', 545--585, Progr. Math., 323, Birkh\"auser/Springer, Cham, 2017.

\bibitem{Schweitzer-92} L. B. Schweitzer. \emph{A short proof that $M_n(A)$ is local if $A$ is local and Fr\'echet}. Internat. J. Math. 3 (1992), no. 4, 581--589.

\bibitem{Solleveld} M. Solleveld. \emph{On the Baum-Connes conjecture with coefficients for linear algebraic groups.} \href{https://arxiv.org/pdf/1901.08807v2.pdf}{https://arxiv.org/pdf/1901.08807v2.pdf}


\bibitem{Valette-84} A. Valette. \emph{Minimal projections, integrable representations and property (T).}  Arch. Math. (Basel) 43 (1984), no. 5, 397--406.

\bibitem{Vigneras-90} M.-F. Vign\'eras. \emph{On formal dimensions for reductive $p$-adic groups.} in ``Festschrift in honor of I. I. Piatetski-Shapiro on the occasion of his sixtieth birthday'', Part I (Ramat Aviv, 1989), 225--266, Israel Math. Conf. Proc., 2, Weizmann, Jerusalem, 1990.


\bibitem{Waldspurger-90} J.-L. Waldspurger. \emph{Demonstration d'une conjecture de dualite de Howe dans le cas $p$-adique, $p\not= 2$}, in ``Festschrift in honor of I. I. Piatetski-Shapiro on the occasion of his sixtieth birthday'', part I (Ramat Aviv, 1989), Israel Mathematical Conference Proceedings, vol. 2 (Weizmann, Jerusalem, 1990), 267--324.


\bibitem{Waldspurger-03} J.-L. Waldspurger. \emph{La formule de Plancherel pour les groupes p-adiques (d'apr\'es Harish-Chandra)}. J. Inst. Math. Jussieu 2 (2003), no. 2, 235--333.


\bibitem{Zettl-82} H. H. Zettl. \emph{Ideals in Hilbert modules and invariants under strong Morita equivalence of $C^*$-algebras.} Arch. Math. (Basel) 39 (1982), no. 1, 69--77.
\end{thebibliography}
\end{document}